\documentclass[UTF8,a4paper,10pt]{article}
\usepackage[left=2.50cm, right=2.50cm, top=2.50cm, bottom=2.50cm]{geometry}
\usepackage{lipsum}
\usepackage{amsmath}
\usepackage{amsthm}
\usepackage{amsfonts}
\usepackage{indentfirst}
\usepackage[noend]{algpseudocode}
\usepackage{algorithmicx,algorithm}
\usepackage{graphicx}
\usepackage{subfigure}
\usepackage{footmisc} 
\usepackage{booktabs}
\usepackage{authblk}   
\usepackage{lineno}
\usepackage{bm}
\usepackage{color}
\usepackage{amssymb}
\title{\bf A positivity-preserving and convergent numerical scheme for the binary fluid-surfactant system}
\date{}
\author[a,b]{Yuzhe Qin}
\affil[a]{Research Center for Mathematics and Mathematics Education, Beijing Normal University at Zhuhai, 519087, China}
\affil[b]{Laboratory of Mathematics and Complex Systems (Ministry of Education), 
School of Mathematical Sciences, Beijing Normal University, Beijing 100875, People's Republic of China}
\author[c]{Cheng Wang\thanks{Corresponding author.}}
\affil[c]{Mathematics Department, University of Massachusetts, North Dartmouth, MA 02747, USA}
\author[b]{Zhengru Zhang} 

\newtheorem{theorem}{Theorem}[section]
\newtheorem{lemma}{Lemma}[section]
\newtheorem{remark}{Remark}[section]

\numberwithin{equation}{section}
\numberwithin{equation}{section}
\begin{document}
\maketitle
\newcommand\blfootnote[1]
{
\begingroup 
\renewcommand\thefootnote{}\footnote{#1}
\addtocounter{footnote}{-1}
\endgroup 
}
\blfootnote{Email addresses: yzqin@bnu.edu.cn(Y. Qin), cwang1@umassd.edu (C. Wang), 
zrzhang@bnu.edu.cn(Z. Zhang)}
\begin{abstract}
In this paper, we develop a first order (in time) numerical scheme for the binary fluid surfactant phase field model. The free energy contains a double-well potential, a nonlinear coupling entropy and a Flory-Huggins potential. The resulting coupled system consists of two Cahn-Hilliard type equations. This system is solved numerically by finite difference spatial approximation, in combination with convex splitting temporal discretization. We prove the proposed scheme is unique solvable, positivity-preserving and unconditionally energy stable. In addition, an optimal rate convergence analysis is provided for the proposed numerical scheme, which will be the first such result for the binary fluid-surfactant system. Newton iteration is used to solve the discrete system. Some numerical experiments are performed to validate the accuracy and energy stability of the proposed scheme.
\end{abstract}


\section{Introduction}
Two important characteristics of surfactants in binary fluid is that they can move towards the fluid interface due to their amphiphilic structure and they can reduce the interfacial tension and system energy \cite{Myers2005Surfactant}. Therefore, surfactants have various applications in many fields such as biotechnology and industry because of their features~\cite{Branger2002Accelerated,Liu2010Phase}. In the past two decades, there have been a number of excellent studies related to models with surfactants~\cite{Fonseca2007Surfactants, Komura1997Two, Laradji1992The, Sman2006Diffuse, Xiao2018Numerical, Yang2017Linear}. Often, there are two different ideas to model the interfacial dynamics with surfactants. One is the sharp interface model which has a long history dated back to one century ago~\cite{Gibbs1874Transactions,Stefan1970Uber}, and this kind of model has been adopted in \cite{James2004A,Khatri2014An}. In fact, sharp interface models have made great progresses in explaining kinetics of diffusional phase transformations and simulating multiphase systems with surfactants at one time. However, there are some difficulties stemming from the interface interactions with various complex processes during the course of phase transformations \cite{Liu2010Phase}. The other approach method is the known phase-field method~\cite{Fix1983Phase, Langer1986Models, Rayleigh1892On, Waals1979The}. This approach makes use of an appropriate free energy functional to character the interfacial dynamics, and it has been adopted to investigate the interfacial dynamics with surfactants in~\cite{Fonseca2007Surfactants, Laradji1992The, Teigen2011A, Teramoto2001Droplet, Theissen1999Lattice}. In particular, phase-field method was used to in \cite{Laradji1992The} to study the phase transition behaviors of the monolayer microemulsion system, formed by surfactant molecules. Generally, the free energy of binary fluid-surfactant model consists of the following two parts: the first part is the classical Ginzburg-Landau double well potential, which is used to describe a binary mixture, and the other part, called nonlinear coupling entropy term, has a historical evolution process, and is used to account for the influence of the surfactant in boosting the formation of interfaces. Laradji et al. in the pioneering work \cite{Laradji1992The} introduced two phase field variables to represent the local densities of the fluids, as well as the local concentration of the surfactant, respectively. As mentioned in~\cite{Komura1997Two}, an extra diffusion term was added to prevent the model from becoming unbounded and a Ginzburg-Landau type potential for the concentration variable to allow the coexistence of the two bulk states. In order to restrict the range of the concentration variable, the authors added the logarithmic Flory-Huggins potential in \cite{Sman2006Diffuse}, based on the nonlinear coupling entropy similar to~\cite{Komura1997Two, Laradji1992The}. In consideration of penalizing the concentration to accumulate along the fluid interface, the authors changed the nonlinear coupled entropy slightly in~\cite{Fonseca2007Surfactants}.  In addition, a further modified model was considered in~\cite{Teng2013Simulating} by adding the Flory-Huggins potential for the local concentration variable as well, in comparison with the model in~\cite{Fonseca2007Surfactants}. 

In this paper, we focus on constructing unconditionally energy stable numerical schemes for the binary fluid-surfactant model in \cite{Teng2013Simulating}. 
There have been some works about numerical approximation to multi-phase models~\cite{Diegel2017A,Ji2019A}. 
Owing to the stiff nonlinear terms originated from the thin interface thickness parameter, 
there are a lot of subtle difficulties to construct numerical schemes with unconditional energy stability, 
especially for the second order accurate (in time) scheme. Lots of efforts have been made to solve these problems~\cite{Diegel2017A, Yan2018A, Zhang2019Linear}, etc. 
Since a simple fully implicit or explicit type discretization brings extremely severe time step size constraint on the interfacial width~
\cite{Boyer2009Numerical, Feng2003Numerical, Shen2010Numerical},  a semi-implicit method was adopted in \cite{Teng2013Simulating}. 
However, the author mentioned that it suffers from a small CFL conditional number. 
Recently, Gu et al. in \cite{Gu2014An} constructed an energy stable finite difference scheme for the binary fluid-surfactant system, 
which is based on the convex splitting approach~\cite{Eyre1998Un, Jie2012Second, Wang2011Energy, Wu2017Stabilized}: 
implicit treatment for the convex part and explicit treatment for the concave part. 
Meanwhile, it is observed that, the convexity analysis for one mixed term has not been theoretically justified in \cite{Gu2014An}, due to the multi variables involved in the system. 
In addition, the positivity-preserving property has not been theoretically proved, so that the well-defined nature of the numerical scheme is not available. 
More recently, Yang et al. constructed the linear and stable schemes for the binary fluid-surfactant system with constant mobility in~\cite{Yang2017Linear}, 
using the invariant energy quadratization (IEQ) technique~
\cite{Cheng2017Efficient, Guill2013On, Han2017Numerical, Yang2016Linear, Yang2017Efficient, Yang2017Numerical, Yang2017Numerical2, Jia2017Numerical}. 
In this approach, the free energy is transformed into an equivalent quadratic form by introducing appropriate auxiliary variables, 
and all nonlinear terms in this system are treated semi-explicitly~\cite{Yang2017Linear}. 
The energy stability has been derived for the IEQ method, while such a stability has to be based on an alternate energy functional (involved with auxiliary variables), 
not for the original energy functional, as always in the IEQ approach. Moreover, the positivity-preserving property is not available to the IEQ-based numerical method, 
because of the explicit treatment for the nonlinear logarithmic term. 
In addition to the IEQ idea, Zhu et al. proposed the scalar auxiliary variable (SAV) method to the surfactant model in \cite{Zhu2018Decoupled}, 
following similar ideas in~\cite{Jie2017A, Jie2018The}. The SAV approach introduces a constant-coefficient linear equation to solve at each time step, 
and the energy stability could be derived for an alternate energy involved with a scalar variable. The convergence and error estimate for the SAV approach, 
for the typical Cahn-Hilliard equation with double-well potential, has also been established in recent works~\cite{Xiao2018Energy, Jie2018Convergence}. 
However, an application of the SAV approach to the surfactant model could not overcome the difficulty to theoretically justify the positivity-preserving property, 
due to the explicit treatment of the logarithmic term. In turn, the Flory-Huggins energy potential has to be re-defined and extended around and beyond the singular phase variable values. 
Also see a more recent work~\cite{Qin2019a} of SAV-based numerical algorithm for the surfactant model. 

Among the existing numerical methods, different approaches have different advantages. Here, we pay attention to the the convex splitting approach, originated from the pioneering work of Eyre \cite{Eyre1998Un}. The idea is that the energy admits a splitting into purely convex and concave parts, that is, $E=E_{c}-E_{e}$, where $E_{c}$ and $E_{e}$ are both convex. Such an idea has had wide applications in various gradient flow models,  including either first or second order accurate schemes. See the related works for the phase field crystal (PFC) and the modified PFC (MPFC) equation \cite{Wise2009An, Wang2010Global}, the epitaxial thin film growth models~\cite{Chen2014A, Feng2018A}, and the Cahn-Hilliard flow coupled with fluid motion~\cite{ChenConvergence, DiegelConvergence}, etc. 

Meanwhile, there have been extensive works of linear numerical schemes for the Cahn-Hilliard and epitaxial thin film equations~\cite{LiD2017, LiD2017b, LiD2016a}, in which stabilized implicit-explicit approach has been applied, and energy stability has been theoretically provided. In the case of a Flory-Huggins energy potential with singular logarithmic terms, the positivity-preserving property has been recently established in~\cite{Du2021a} for the corresponding Allen-Cahn equation, based on the maximum principle arguments. The advantage of such a linear scheme (corresponding to an explicit treatment of the nonlinear logarithmic term) is associated with the computational efficiency, so that a nonlinear Newton iteration is not required. On the other hand, this approach works very well for the positivity preserving analysis for the Allen-Cahn gradient flow, due to the availability of maximum principle, while its direct extension to the Cahn-Hilliard gradient flow would face a serious theoretical difficulty. In this paper, we design a uniquely solvable, positivity-preserving, unconditionally energy stable, and first order in time convergent scheme for the binary fluid-surfactant system, based on the convex-splitting idea, combined with the centered difference spatial approximation. For the theoretical analysis of the positivity-preserving property, we make use of the singular nature of the logarithmic function, and prove that such a singular nature prevents the numerical solution approaches the singular limit values, following similar ideas of in the analysis for the Cahn-Hilliard model~\cite{Chen2019Pos, dong20b, Wang2019A, DOng2019A}, as well as the one for the Poisson-Nernst-Planck system~\cite{LiuC20a, Qian20a}, droplet liquid film model~\cite{ZhangJ2021}, etc. In addition, an optimal rate convergence analysis is provided, which is the first such work for the surfactant model. The key difficulty in such an analysis is associated with the logarithmic potential term and the coupled term. In this article, we can make full use of the convexity of energy associated with the nonlinear terms to directly deal with all logarithmic terms and coupled terms, because the convexity of energy indicates the corresponding nonlinear error inner product is always non-negative. 

The rest of the paper is organized as follows. In Section~\ref{sec: model}, we give a brief introduction to the binary fluid-surfactant phase field model and state its energy law. 
In Section~\ref{sec: numerical scheme}, the numerical scheme is proposed and analyzed, and we prove the unique solvability, positivity-preserving property, 
as well as the energy stability. An optimal rate convergence estimate is also provided. 
Some numerical experiments are presented in Section~\ref{sec: numerical results}. 
Finally, some conclusions are made in Section~\ref{sec: conclusion} .
\section{The mathematical model: binary fluid-surfactant system} \label{sec: model}
In this paper, we consider the two-dimensional (2-D) binary fluid-surfactant system. 
With the domain given by $\Omega=\left(0,L_{x}\right)\times\left(0,L_{y}\right)$, the binary fluid-surfactant system is formulated as 
\begin{subequations}\label{bfs}
\begin{align}
\phi_{t} &=M_{1}\Delta \mu_{\phi},\label{eqn_phi}\\
\rho_{t} &= M_{2}\nabla\cdot \left(M\left(\rho\right)\nabla \mu_{\rho}\right),\label{eqn_rho}\\
\mu_{\phi} &= \frac{\delta G}{\delta \phi} = \frac{f' (\phi)}{\varepsilon} - \varepsilon \Delta \phi 
+ \alpha \nabla \cdot \Big( \frac{ \rho \nabla \phi }{| \nabla \phi |} \Big) ,\label{eqn_mu_phi}\\
\mu_{\rho} &= \frac{\delta G}{\delta \rho} = - \alpha | \nabla \phi | + \beta H' (\rho) ,\label{eqn_mu_rho}
\end{align}
\end{subequations}
with the periodic boundary condition and $M\left(\rho\right)=\rho\left(1-\rho\right)$.
The PDE system~\eqref{bfs} corresponds to the following free energy functional
\begin{equation} \label{bfs-energy-1} 
G\left(\phi,\rho\right) = \int_{\Omega}\Big(\frac{f\left( \phi\right)}{\varepsilon}
+\frac{\varepsilon}{2}|\nabla \phi|^{2}+\frac{\alpha}{2}\left(\rho-|\nabla \phi|\right)^{2}
+\beta H\left(\rho\right)\Big)\mathrm{d}\bm{x},
\end{equation}
where 
\begin{equation*}
f\left(\phi\right) = \frac{1}{4}\phi^{2}\left(1-\phi\right)^{2},\quad 
H\left(\rho\right) = \rho\ln\rho+\left(1-\rho\right)\ln\left(1-\rho\right) , 
\end{equation*}
and $\alpha,\beta, \varepsilon$ are all small positive parameters. In this paper we assume $M\left(\rho\right)=1$ and $M_{1}=M_{2}=\mathcal{M}$ for simplicity. 
Furthermore, to avoid the singularity in calculating the coupled energy $\left(\rho-|\nabla \phi|\right)^{2}$,  
we use $\sqrt{\phi_{x}^{2}+\phi_{y}^{2}+\delta^{2}}$ to approximate $|\nabla \phi|$. 
And also, we add diffuse terms $\frac{\eta^{2}}{2}|\Delta\phi|^{2}$ and $\frac{\xi}{2}|\nabla\rho|^{2}$ in the energy density, so that the new free energy functional becomes 
\begin{equation}\label{energy}
G_{new}\left(\phi,\rho\right) = \int_{\Omega}
\Big(\frac{f\left( \phi\right)}{\varepsilon}
+\frac{\varepsilon}{2}|\nabla \phi|^{2}
+\frac{\eta^{2}}{2}|\Delta\phi|^{2}
+\frac{\xi}{2}|\nabla\rho|^{2}
+\frac{\alpha}{2}\left(\rho-|\nabla \phi|\right)^{2}
+\beta H\left(\rho\right)\Big)\mathrm{d}\bm{x}.
\end{equation}
For simplicity, we still use $G$ to express $G_{new}$. In turn, the corresponding chemical potentials become  
\begin{subequations}\label{bfs-new}
\begin{align}
\mu_{\phi} &= \frac{\delta G}{\delta \phi} = \frac{f' (\phi)}{\varepsilon} - \varepsilon \Delta \phi 
+ \eta^2 \Delta^2 \phi + \alpha \nabla \cdot \Big( \frac{ \rho \nabla \phi }{| \nabla \phi |} \Big) ,\label{eqn_mu_phi-new}\\
\mu_{\rho} &= \frac{\delta G}{\delta \rho} = - \alpha | \nabla \phi | - \xi \Delta \rho + \beta H' (\rho) .\label{eqn_mu_rho-bew}
\end{align}
\end{subequations}

As always in the gradient system, the energy dissipation property is always valid: 
\begin{equation*}
\frac{\mathrm{d}}{\mathrm{d}t}G\left(\phi\left(t\right),\rho\left(t\right)\right)\leq 0,\quad t>0.
\end{equation*}
Besides, the appearance of the Flory-Huggins energy indicates a positivity property for the density variable, $0 < \rho <1$ at a point-wise level. Our primary aim is to develop a numerical scheme inheriting these properties at a theoretical level.

\begin{remark} 
Typically, in the context of physical models, the Dirichlet energy $\| \nabla u \|^2$ represents surface tension, whereas the higher order term  $\| \Delta u \|^2$ represents bending rigidity. In general, it may be assumed that all orders of the energy density are represented in the expansion of the energy  
$$
  E (u) = f(u) + a_0 u^2 + a_2 \| \nabla u \|^2 + a_4 \| \Delta u \|^2 + ... , 
$$
where f may be non-quadratic, and the coefficients $a_k$ may depend upon $u$, or derivatives of $u$ but are usually constants. On the other hand, it is typical to drop all higher order terms beyond those that are physically/mathematically necessary to make the PDE well posed. One would assume that the corresponding coefficients are sufficiently small so as to justify this. In most cases, it may be more reasonable to assume that $a_2 = 0$ (or is negligible) and only keep terms of order higher than two.   

For the binary fluid-surfactant system~\eqref{bfs}, combined with the physical energy~\eqref{bfs-energy-1}, we are able to construct a numerical scheme with an energy stability, while the optimal rate convergence analysis will face essential theoretical difficulties, due to the highly nonlinear and singular 1-Laplacian term involved for the variable $\phi$. To overcome this subtle difficulty, we add an additional bi-harmonic diffusion term for $\phi$, as well as a regular diffusion term for the $\rho$, in the energy representation~\eqref{energy}. As a result, both the energy stability and optimal rate convergence analysis could be theoretically justified, as will be demonstrated in the later section. In particular, an optimal rate convergence analysis will be the first such result for the binary fluid-surfactant system. Such an approach of adding higher order bi-harmonic diffusion process has been reported in many related nonlinear physical systems, in which the diffusion terms have played essential roles in the stability analysis, such as the planetary geostrophic equations of oceanic geophysical fluid model~\cite{LiuJ2006, STWW07, STW98, STW00}, etc.  
\end{remark}

\section{The numerical scheme} \label{sec: numerical scheme}

In this section, we present a convex-concave decomposition of the energy $\eqref{energy}$, and propose a convex splitting scheme based on such a decomposition. 
The unique solvability, energy stability, positivity-preserving property will be analyzed afterward.  

\subsection{The convex-concave decomposition of the energy}
\begin{lemma} 
Suppose that $\Omega=\left(0,L_{x}\right)\times\left(0,L_{y}\right)$ and $\phi,\rho: \Omega \to \mathbb{R}$ are periodic and sufficiently regular. 
Define the following energy functionals 
\begin{align*}
G_{c}\left(\phi,\rho\right)
=& \int_{\Omega}\frac{1}{4\varepsilon} (\phi-\frac{1}{2} )^{4}+\frac{1}{64\varepsilon}
    +\frac{\varepsilon}{2}|\nabla \phi|^{2}+\frac{\eta^{2}}{2}|\Delta\phi|^{2}
    +\beta H\left(\rho\right)+\frac{\xi}{2}|\nabla\rho|^{2}\\
    &+\frac{\alpha}{2}\left\{\left(\rho-|\nabla^\delta \phi|\right)^{2}
    + (\sqrt{2}-1 )\rho^{2}+\frac{1}{\delta}|\nabla \phi|^{2}\right\}\mathrm{d}\bm{x},\\
G_{e}\left(\phi,\rho\right)
=& \int_{\Omega}\frac{1}{8\varepsilon} (\phi-\frac{1}{2} )^{2}
    +\frac{\alpha}{2}\left\{ (\sqrt{2}-1 )\rho^{2}
    +\frac{1}{\delta}|\nabla \phi |^{2}\right\}\mathrm{d}\bm{x} , 
\end{align*}
with $|\nabla^\delta \phi| := ( | \nabla \phi |^2 + \delta^2 )^\frac12$.   
Then $G_{c}\left(\phi,\rho\right)$ and $G_{e}\left(\phi,\rho\right)$ are both convex 
with respect to $\phi$ and $\rho$, with $G\left(\phi,\rho\right)=G_{c}\left(\phi,\rho\right)-G_{e}\left(\phi,\rho\right)$.
\end{lemma}

\begin{proof}
We focus on the convexity analysis of $G_{c}\left(\phi,\rho\right)$ 
and $G_{e}\left(\phi,\rho\right)$. Let 
\begin{align*}
e_{c}\left(\phi,\phi_{x},\phi_{y},\Delta\phi,\rho,\rho_{x},\rho_{y}\right)
=& \frac{1}{4\varepsilon} (\phi-\frac{1}{2} )^{4}+\frac{1}{64\varepsilon}
    +\frac{\varepsilon}{2}|\nabla \phi|^{2}+\frac{\eta^{2}}{2}|\Delta\phi|^{2}
    +\beta H\left(\rho\right)+\frac{\xi}{2}|\nabla\rho|^{2}\\
    &+\frac{\alpha}{2}\left\{\left(\rho-|\nabla \phi|\right)^{2}
    + (\sqrt{2}-1 )\rho^{2}+\frac{1}{\delta}|\nabla \phi|^{2}\right\},\\
e_{e}\left(\phi,\phi_{x},\phi_{y},\Delta\phi,\rho,\rho_{x},\rho_{y}\right)
=& \frac{1}{8\varepsilon} (\phi-\frac{1}{2} )^{2}
    +\frac{\alpha}{2}\left\{ (\sqrt{2}-1 )\rho^{2}+\frac{1}{\delta}|\nabla \phi|^{2}\right\} . 
\end{align*}
We also denote 
\begin{subequations}
\begin{align*}
e_{c_{1}}\left(\bm{v}\right)
& \triangleq e_{c1}\left(\phi,\phi_{x},\phi_{y},\Delta\phi,\rho,\rho_{x},\rho_{y}\right)
    =\frac{1}{4\varepsilon} (\phi-\frac{1}{2} )^{4}+\frac{1}{64\varepsilon}
    +\frac{\varepsilon}{2}|\nabla \phi|^{2}+\frac{\eta^{2}}{2}|\Delta\phi|^{2}
    +\beta H\left(\rho\right)+\frac{\xi}{2}|\nabla\rho|^{2},\\
e_{c_{2}}\left(\bm{v}\right)
& \triangleq e_{c2}\left(\phi,\phi_{x},\phi_{y},\Delta\phi,\rho,\rho_{x},\rho_{y}\right)
    =\frac{\alpha}{2}\left\{\left(\rho-|\nabla^\delta \phi|\right)^{2}
    + (\sqrt{2}-1 )\rho^{2}+\frac{1}{\delta}|\nabla \phi|^{2}\right\},\\
e_{e}\left(\bm{v}\right)
&\triangleq e_{e}\left(\phi,\phi_{x},\phi_{y},\Delta\phi,\rho,\rho_{x},\rho_{y}\right)
    =\frac{1}{8\varepsilon} (\phi-\frac{1}{2} )^{2}
    +\frac{\alpha}{2}\left\{ (\sqrt{2}-1 )\rho^{2}+\frac{1}{\delta}|\nabla \phi|^{2}\right\},
\end{align*}
\end{subequations}
where 
\begin{equation*}
{\bm{v}}=\left(v_{1},v_{2},v_{3},v_{4},v_{5},v_{6},v_{7}\right)
\triangleq\left(\phi,\phi_{x},\phi_{y},\Delta\phi,\rho,\rho_{x},\rho_{y}\right) . 
\end{equation*} 
Then we have 
\begin{equation*}
e_{c}\left(\bm{v}\right)=e_{c_{1}}\left(\bm{v}\right)+e_{c_{2}}\left(\bm{v}\right), \quad 
G_{c}\left(\phi,\rho\right)=\int_{\Omega}e_{c}\left(\bm{v}\right)\mathrm{d}\bm{x}, \quad 
G_{e}\left(\phi,\rho\right)=\int_{\Omega}e_{e}\left(\bm{v}\right)\mathrm{d}\bm{x},
\end{equation*}
and the following inequalities are derived: 
\begin{subequations}
\begin{align*}
\partial_{v_{1}}^{2}e_{c_{1}}\left(v_1,v_{2},v_{3},v_{4},v_{5},v_{6},v_{7}\right)
&=\frac{3}{\varepsilon} (\phi-\frac{1}{2} )^{2}\geq 0,\\
\partial_{v_{2}}^{2}e_{c_{1}}\left(v_1,v_{2},v_{3},v_{4},v_{5},v_{6},v_{7}\right)&=\varepsilon>0,\\
\partial_{v_{3}}^{2}e_{c_{1}}\left(v_1,v_{2},v_{3},v_{4},v_{5},v_{6},v_{7}\right)&=\varepsilon>0,\\
\partial_{v_{4}}^{2}e_{c_{1}}\left(v_1,v_{2},v_{3},v_{4},v_{5},v_{6},v_{7}\right)&=\eta^{2}> 0,\\
\partial_{v_{5}}^{2}e_{c_{1}}\left(v_1,v_{2},v_{3},v_{4},v_{5},v_{6},v_{7}\right)
&=\frac{\beta}{\rho\left(1-\rho\right)}>0,\\
\partial_{v_{6}}^{2}e_{c_{1}}\left(v_1,v_{2},v_{3},v_{4},v_{5},v_{6},v_{7}\right)&=\xi> 0,\\
\partial_{v_{7}}^{2}e_{c_{1}}\left(v_1,v_{2},v_{3},v_{4},v_{5},v_{6},v_{7}\right)&=\xi> 0,\\
\partial_{v_{1}}^{2}e_{e}\left(v_1,v_{2},v_{3},v_{4},v_{5},v_{6},v_{7}\right)&=\frac{1}{4\varepsilon}>0,\\
\partial_{v_{2}}^{2}e_{e}\left(v_1,v_{2},v_{3},v_{4},v_{5},v_{6},v_{7}\right)&=\frac{\alpha}{\delta}>0,\\
\partial_{v_{3}}^{2}e_{a}\left(v_1,v_{2},v_{3},v_{4},v_{5},v_{6},v_{7}\right)&=\frac{\alpha}{\delta}>0,\\
\partial_{v_{4}}^{2}e_{c_{1}}\left(v_1,v_{2},v_{3},v_{4},v_{5},v_{6},v_{7}\right)&=0,\\
\partial_{v_{5}}^{2}e_{e}\left(v_1,v_{2},v_{3},v_{4},v_{5},v_{6},v_{7}\right)&= (\sqrt{2}-1 )\alpha>0.
\end{align*}
\end{subequations}
These facts imply that both $e_{c_{1}}\left(\bm{v}\right)$ and $e_{e}\left(\bm{v}\right)$ are convex. 
To get the convexity of $e_{c_{2}}\left(\bm{v}\right)$, we have to analyze the Hessian matrix of $e_{c_{2}}\left(\hat{\bm{v}}\right)$ as follows, 
where $\hat{\bm{v}}=\left(v_{2},v_{3},v_{5}\right)$: 
\begin{equation}
H\left(\hat{\bm{v}}\right)=\alpha\left[
\begin{matrix}
\displaystyle
\sqrt{2}& -\frac{v_{2}}{\sqrt{v_{2}^{2}+v_{3}^{2}+\delta^2}} 
& -\frac{v_{3}}{\sqrt{v_{2}^{2}+v_{3}^{2}+\delta^2}}\\
-\frac{v_{2}}{\sqrt{v_{2}^{2}+v_{3}^{2}+\delta^2}} 
& -\frac{v_{5}\left(v_3^{2}+\delta^2\right)}{\left(v_{2}^{2}+v_{3}^{2}+\delta^{2}\right)^{\frac{3}{2}}}
+\frac{1}{\delta}+1 &\frac{v_{2}v_{3}v_{5}}{\left(v_{2}^{2}+v_{3}^{2}+\delta^{2}\right)^{\frac{3}{2}}}\\
-\frac{v_{3}}{\sqrt{v_{2}^{2}+v_{3}^{2}+\delta^2}} 
&\frac{v_{2}v_{3}v_{5}}{\left(v_{2}^{2}+v_{3}^{2}+\delta^{2}\right)^{\frac{3}{2}}}
&-\frac{v_{5}\left(v_2^{2}+\delta^2\right)}{\left(v_{2}^{2}+v_{3}^{2}+\delta^{2}\right)^{\frac{3}{2}}}
+\frac{1}{\delta}+1
\end{matrix}
\right] .  \label{Hessian-H-1} 
\end{equation}
Since $0 < \rho < 1$, i.e. $v_{5}\in\left(0,1\right)$, a careful application of calculus reveals that $H\left(\hat{\bm{v}}\right)$ is diagonally dominated, so that it is non-negative definite. This in turn indicates the convexity of $e_{c_{2}}$. Therefore, we obtain the following inequality, according to the definition of convex function: 
\begin{subequations}
\begin{align}
&e_{c}\left(\lambda\bm{w}+\left(1-\lambda\right)\bm{v}\right)\leq 
\lambda e_{c}\left(\bm{w}\right)+\left(1-\lambda\right)e_{c}\left(\bm{v}\right),\label{convex_c}\\
&e_{e}\left(\lambda\bm{w}+\left(1-\lambda\right)\bm{v}\right)\leq 
\lambda e_{e}\left(\bm{w}\right)+\left(1-\lambda\right)e_{e}\left(\bm{v}\right),\label{convex_a}
\end{align}
\end{subequations}
where $\lambda\in\left(0,1\right), \bm{w}, \bm{v}\in\mathbb{R}^{7}.$
Integrating both sides of $\eqref{convex_c}$ and $\eqref{convex_a}$ leads to 
\begin{equation*}
G_{c}\left(\lambda \phi_{1}+\left(1-\lambda\right) \phi_{2}, 
\lambda \rho_{1}+\left(1-\lambda\right) \rho_{2}\right)\leq 
\lambda G_{c}\left(\phi_{1},\rho_{1}\right)
+\left(1-\lambda\right)G_{c}\left(\phi_{2},\rho_{2}\right),
\end{equation*}
and 
\begin{equation*}
G_{e}\left(\lambda\phi_{1}+\left(1-\lambda\right) \phi_{2}, 
\lambda \rho_{1}+\left(1-\lambda\right) \rho_{2}\right)\leq 
\lambda G_{e}\left(\phi_{1},\rho_{1}\right)
+\left(1-\lambda\right)G_{e}\left(\phi_{2},\rho_{2}\right),
\end{equation*}
which indicates that both $G_{c}\left(\phi,\rho\right)$ and $G_{e}\left(\phi,\rho\right)$ 
are convex with respect to $\phi$ and $\rho$.
\end{proof}

As a generalization of the theorem presented in~\cite{Wise2009An}, the following lemma is the foundation of energy stability for binary fluid-surfactant, 
or more generally, for two variable functional. The proof is similar to \cite{Wise2009An}, so we skip it for the sake of brevity.
 
\begin{lemma}
Assume $\phi,\varphi,\rho,\psi:\Omega\to\mathbb{R}$ are periodic and smooth enough.
If $G=G_{c}-G_{e}$ gives a convex-concave decomposition, then we have 
\begin{equation}\label{une}
G\left(\phi,\rho\right)-G\left(\varphi,\psi\right)\leq
\left(\delta_{\phi}G_{c}\left(\phi,\rho\right)-\delta_{\phi}G_{e}\left(\varphi,\psi\right),
\phi-\varphi\right)_{L^{2}}+
\left(\delta_{\rho}G_{c}\left(\phi,\rho\right)-\delta_{\rho}G_{e}\left(\varphi,\psi\right),
\rho-\psi\right)_{L^{2}},
\end{equation}
where $\delta$ denotes the variational derivative.
\end{lemma}

Given a time step $\Delta t>0$, we construct the discrete-time, continuous-space scheme of the binary fluid-surfactant system~$\eqref{bfs}$ as follows
\begin{subequations}\label{CS}
\begin{align}
\frac{\phi^{n+1}-\phi^{n}}{\Delta t} &=\mathcal{M}\Delta \mu_{\phi}^{n+1},\\
\frac{\rho^{n+1}-\rho^{n}}{\Delta t} &= \mathcal{M}\Delta \mu_{\rho}^{n+1},\\
\mu_{\phi}^{n+1} 
&= \delta_{\phi} G_{c}\left(\phi^{n+1},\rho^{n+1}\right)-\delta_{\phi}G_{e}\left(\phi^{n},\rho^{n}\right),\\
\mu_{\rho}^{n+1} 
&= \delta_{\rho} G_{c}\left(\phi^{n+1},\rho^{n+1}\right)-\delta_{\rho}G_{e}\left(\phi^{n},\rho^{n}\right).
\end{align}
\end{subequations}

\subsection{The spatial discretization and the fully discrete numerical scheme}
The centered difference approximation is applied to discretize the space. 
Here we first recall some basic notations of this methodology, and we use the similar notations and results for some discrete functions and operators introduced in~
\cite{Guo2016An, Wise2009An}. Let $\Omega=(0,L_x)\times(0,L_y)$, where for simplicity, we assume $L_x=L_y:=L>0$. Let $N\in \mathbb{N}$ be given, 
and define the grid mesh size $h:=L/N$. Such a uniform mesh size assumption is only for simplicity of presentation.
We define the following two uniform, infinite node sets with grid spacing $h>0$:
\begin{equation}
E:=\{p_{i+\frac{1}{2}}~|~i\in\mathbb{Z}\},\quad C:=\{p_{i}~|~i\in\mathbb{Z}\},
\end{equation}
where $p_i=p(i):=(i-\frac{1}{2})\cdot h$. 
Consider the following 2-D discrete $N^{2}$-peroidic function spaces:
\begin{align*}
&\mathcal{C}_{per}:=\{\nu: C\times C\to\mathbb{R}|\nu_{i,j}
=\nu_{i+\alpha N,j+\beta N},~\forall i,j,\alpha,\beta \in \mathbb{Z}\},\\
&\mathcal{E}_{per}^{x}:=\{\nu: E\times C\to\mathbb{R}|\nu_{i+\frac{1}{2},j}
=\nu_{i+\frac{1}{2}+\alpha N,j+\beta N},~\forall i,j,\alpha,\beta \in \mathbb{Z}\}.
\end{align*}
The spaces $\mathcal{E}_{per}^{y}$ can be analogously defined. Here we use the Greek symbols $\nu_{i,j}=\nu(p_{i},p_{j})$, et cetera. 
The functions of $\mathcal{C}_{per}$ are called $cell-centered$ functions, while the functions of $\mathcal{E}_{per}^{x}$ and $\mathcal{E}_{per}^{y}$ are called 
$x-direction$ and $y-direction$ $edge-centered$ functions, respectively. 
We also define the mean zero space
\begin{equation*}
\mathring{\mathcal{C}}_{per}:=\left\{\nu \in \mathcal{C}_{per}~|~
0=\bar{\nu}:=\frac{h^{2}}{|\Omega|}\sum_{i,j=1}^{m}\nu_{i,j}\right\}.
\end{equation*}
Additionally, we denote $\vec{\mathcal{E}}_{per}:=\mathcal{E}_{per}^{x} \times \mathcal{E}_{per}^{y}$.

Next, the important difference and average operators are introduced on the function spaces:
\begin{align*}
&A_{x}\nu_{i+\frac{1}{2},j}:=\frac{1}{2}(\nu_{i+1,j}+\nu_{i,j}),\quad 
 D_{x}\nu_{i+\frac{1}{2},j}:=\frac{1}{h}(\nu_{i+1,j}-\nu_{i,j}),\\
&A_{y}\nu_{i,j+\frac{1}{2}}:=\frac{1}{2}(\nu_{i,j+1}+\nu_{i,j}),\quad
 D_{y}\nu_{i,j+\frac{1}{2}}:=\frac{1}{h}(\nu_{i,j+1}-\nu_{i,j}),
\end{align*}
with 
$A_{x},D_{x}:\mathcal{C}_{per}\to \mathcal{E}_{per}^{x}$ and 
$A_{y},D_{y}:\mathcal{C}_{per}\to \mathcal{E}_{per}^{y}$. Likewise,
\begin{align*}
&a_{x}\nu_{i,j}:=\frac{1}{2}(\nu_{i+\frac{1}{2},j}+\nu_{i-\frac{1}{2},j}),\quad 
 d_{x}\nu_{i,j}:=\frac{1}{h}(\nu_{i+\frac{1}{2},j}-\nu_{i-\frac{1}{2},j}),\\
&a_{y}\nu_{i,j}:=\frac{1}{2}(\nu_{i,j+\frac{1}{2}}+\nu_{i,j-\frac{1}{2}}),\quad
 d_{y}\nu_{i,j}:=\frac{1}{h}(\nu_{i,j+\frac{1}{2}}-\nu_{i,j-\frac{1}{2}}),
\end{align*}
with $a_{x},d_{x}:\mathcal{C}_{per}\to \mathcal{C}_{per}^{x}$ and $a_{y},d_{y}:\mathcal{C}_{per}\to \mathcal{C}_{per}^{y}$. 
The discrete gradient $\nabla_{h}:\mathcal{C}_{per}\to\vec{\mathcal{E}}_{per}$ is defined via 
\begin{equation*}
\nabla_{h}\nu_{i,j}=(D_{x}\nu_{i+\frac{1}{2},j},D_{y}\nu_{i,j+\frac{1}{2}}),
\end{equation*}
and the discrete divergence $\nabla_{h}\cdot:\mathcal{E}_{per}\to\vec{\mathcal{C}}_{per}$ becomes 
\begin{equation*}
\nabla_{h}\cdot \vec{f}_{i,j}=d_{x}f_{i,j}^x +d_{y }f_{i,j}^y , \quad 
\mbox{for} \, \, \, \vec{f}=(f^{x},f^{y}) \in \vec{\mathcal{E}}_{per} . 
\end{equation*}
The standard 2-D discrete Laplacian, $\Delta_{h}:\mathcal{C}_{per}\to\mathcal{C}_{per}$, is given by
\begin{align*}
\Delta_{h}\nu_{i,j}
:&=\nabla_{h}\cdot(\nabla_{h} \nu)_{i,j}
  =d_{x}(D_{x}\nu)_{i,j}+d_{y}(D_{y}\nu)_{i,j}\\
 &=\frac{1}{h^{2}}(\nu_{i+1,j}+\nu_{i-1,j}+\nu_{i,j+1}+\nu_{i,j-1}-4\nu_{i,j}).
\end{align*}
More generally, if $\mathcal{D}$ is a periodic $scalar$ function that is defined at all of the face center points and $\vec{f} \in \vec{\mathcal{E}}_{per}$, 
assuming point-wise multiplication, we may define
\begin{equation*}
\nabla_{h}\cdot(\mathcal{D}\vec{f})_{i,j}
=d_{x}(\mathcal{D}f^{x})_{i,j}+d_{y}(\mathcal{D}f^{y})_{i,j}.
\end{equation*}
Specifically, if $\nu\in\mathcal{C}_{per}$, then 
$\nabla_{h}\cdot(\mathcal{D}\nabla_{h}):\mathcal{C}_{per}\to\mathcal{C}_{per}$ 
is defined point-wise via
\begin{equation*}
\nabla_{h}\cdot(\mathcal{D}\nabla_h\nu)_{i,j}
=d_{x}(\mathcal{D}D_{x}\nu)_{i,j}+d_{y}(\mathcal{D}D_{y})_{i,j}.
\end{equation*}

Now we are ready to introduce the following grid inner products:
\begin{align*}
&\langle\nu,\xi\rangle_{\Omega}:=h^{2}\sum_{i,j=1}^{N}\nu_{i,j}\xi_{i,j},\quad
 \nu, \xi \in \mathcal{C}_{per}, \\
&[\nu,\xi]_{x}:=\langle a_{x}(\nu\xi),1\rangle_{\Omega},\quad 
 \nu, \xi \in \mathcal{E}_{per}^{x},\\
&[\nu,\xi]_{y}:=\langle a_{y}(\nu\xi),1\rangle_{\Omega},\quad 
 \nu, \xi \in \mathcal{E}_{per}^{y}.\\
&[\vec{f}_{1},\vec{f}_{2}]_{\Omega}
:=[f_{1}^{x},f_{2}^{x}]_{x}
 +[f_{1}^{y},f_{2}^{y}]_{y},\quad 
\vec{f}_{i}=(f_{i}^{x},f_{i}^{y}) \in \vec{\mathcal{E}}_{per},\quad i=1,2.
\end{align*}

We define the following norms for cell-centered functions. 
If $\nu\in\mathcal{C}_{per}$, then $\|\nu\|_{2}^{2}:=\langle\nu,\nu\rangle_{\Omega}$; $\|\nu\|_{p}^{p}:=\langle|\nu|^{p},1\rangle_{\Omega}$, for $1\leq p\leq \infty$, 
and $\|\nu\|_{\infty}:=\max\limits_{1\leq i,j\leq N}|\nu_{i,j}|$. The norms of the gradient are defined as follows: for $\nu\in\mathcal{C}_{per}$,
\begin{equation*}
\|\nabla_{h}\nu\|_{2}^{2}
:=[\nabla_{h}\nu,\nabla_{h}\nu]_{\Omega}
 =[D_{x}\nu,D_{x}\nu]_{x}
 +[D_{y}\nu,D_{y}\nu]_{y},
\end{equation*}
and, more generally, for $1\leq p \leq \infty$,
\begin{equation*}
\|\nabla_{h}\nu\|_{p}
:=\left(
[|D_{x}\nu|^{p},1]_{x}+
[|D_{y}\nu|^{p},1]_{y}\right)^{\frac{1}{p}},
\end{equation*}
Higher order norms can be similarly introduced; for example,
\begin{equation*}
\|\nu\|_{H_{h}^{1}}^{2}:=\|\nu\|_{2}^{2}+\|\nabla_{h}\nu\|_{2}^{2},\quad 
\|\nu\|_{H_{h}^{2}}^{2}:=\|\nu\|_{H_{h}^{1}}^{2}+\|\Delta_{h}\nu\|_{2}^{2}.
\end{equation*}

To facilitate the convergence analysis, we need to introduce a discrete analogue of the space $H_{per}^{-1}(\Omega)$, as outlined in \cite{Wang2011Energy}. 
Suppose that $\mathcal{D}$ is a positive, periodic scalar function defined at all of the face center points. 
For any $\phi\in\mathcal{C}_{per}$, there exists a unique $\psi\in\mathring{\mathcal{C}}_{per}$, that solves
\begin{equation}
\mathcal{L}_{\mathcal{D}}(\psi):=-\nabla_{h}\cdot\left(\mathcal{D}\nabla_{h}\psi\right)=\phi-\bar{\phi},
\end{equation}
where, recall, $\bar{\phi}:=|\Omega|^{-1}\langle\phi,1\rangle_{\Omega}$. We equip this space with a bilinear form: for any $\phi_{1},\phi_{2}\in\mathring{\mathcal{C}}_{per}$, define
\begin{equation}
\langle\phi_{1},\phi_{2}\rangle_{\mathcal{L}_{\mathcal{D}}^{-1}}
:=[\mathcal{D}\nabla_{h}\psi_{1},\nabla_{h}\psi_{2}]_{\Omega},
\end{equation}
where $\psi_{i}\in\mathring{\mathcal{C}}_{per}$ is the unique solution to 
\begin{equation}
\mathcal{L}_{\mathcal{D}}(\psi_{i}):=-\nabla_{h}\cdot(\mathcal{D}\nabla_{h}\psi_{i})=\phi_{i},\quad i=1,2.
\end{equation}
The following identity is easy to prove via summation-by-parts:
\begin{equation}
\langle\phi_{1},\phi_{2}\rangle_{\mathcal{L}_{\mathcal{D}}^{-1}}
=\langle\phi_{1},\mathcal{L}_{\mathcal{D}}^{-1}(\phi_{2})\rangle_{\Omega}
=\langle\mathcal{L}_{\mathcal{D}}^{-1}(\phi_{1}),\phi_{2}\rangle_{\Omega},
\end{equation}
and since $\mathcal{L}_{\mathcal{D}}$ is symmetric positive definite, $\langle\cdot,\cdot\rangle_{\mathcal{L}_{\mathcal{D}}^{-1}}$ is 
an inner product on $\mathring{\mathcal{C}}_{per}$. When $\mathcal{D}\equiv 1$, we drop the subscript and write $\mathcal{L}_{1} = \mathcal{L}$, 
and in this case we usually write $\langle\cdot,\cdot\rangle_{\mathcal{L}_{\mathcal{D}}^{-1}}=:\langle\cdot,\cdot\rangle_{-1,h}$. 
In the general setting, the norm associated to this inner product is denoted as 
$\|\phi\|_{\mathcal{L}_{\mathcal{D}}^{-1}}:=\sqrt{\langle\phi,\phi\rangle_{\mathcal{L}_{\mathcal{D}}^{-1}}}$,
 for all $\phi \in \mathring{\mathcal{C}}_{per}$, but, if $\mathcal{D}\equiv 1$, we write $\|\cdot\|_{\mathcal{L}_{\mathcal{D}}^{-1}}=:\|\cdot\|_{-1,h}$.

With the preparations above, we turn to discuss the discrete energy and the fully discrete scheme. 
Define the discrete energies $E,E_{c},E_{e}:\mathcal{C}_{per}\times\mathcal{C}_{per}\to\mathbb{R}$ as
\begin{equation*}
\begin{aligned} 
E\left(\phi,\rho\right)= & 
h^{2}\sum_{i,j=1}^{N} \Big(\frac{f\left(\phi_{i,j}\right)}{\varepsilon}
+\frac{\varepsilon}{2}\left|\nabla_{h}\phi_{i,j}\right|^{2}
+\frac{\eta^{2}}{2}\left|\Delta_{h}\phi_{i,j}\right|^{2}
+\frac{\xi}{2}\left|\nabla_{h}\rho_{i,j}\right|^{2} 
\\
  &  \qquad \quad 
+\frac{\varepsilon}{2}\left(\rho_{i,j}- {\cal A} |\nabla_{h}^\delta \phi|_{i,j} \right)^{2}
+\beta H (\rho_{i,j} ) \Big), 
\end{aligned} 
\end{equation*}
in which ${\cal A} |\nabla_{h}^\delta \phi|$ is defined as 
$$
  ( {\cal A} |\nabla_{h}^\delta \phi|_{i,j} )^2 = \frac12 ( (D_x \phi)_{i+\frac12, j}^2 + (D_x \phi)_{i-\frac12, j}^2 + (D_y \phi)_{i, j+\frac12}^2 + (D_y \phi)_{i, j-\frac12}^2) + \delta^2 . 
$$ 
The relevant discrete convex and concave energy functionals $E_{c}, E_{e}:\mathcal{C}_{per}\times\mathcal{C}_{per}\to\mathbb{R}$ are given by
\begin{align}
E_{c}\left(\phi,\rho\right)
=& h^{2}\sum_{i,j=1}^{N}\left(\frac{1}{4\varepsilon} (\phi_{i,j}-\frac{1}{2} )^{4}
+\frac{1}{64\varepsilon}
    +\frac{\varepsilon}{2}|\nabla_{h} \phi_{i,j}|^{2}
    +\frac{\eta^{2}}{2}|\Delta_{h}\phi_{i,j}|^{2}
    +\beta H\left(\rho_{i,j}\right)
    +\frac{\xi}{2}|\nabla_{h}\rho_{i,j}|^{2}\right.  \nonumber \\
    &\left.+\frac{\alpha}{2}\left(\left(\rho_{i,j}- {\cal A} |\nabla_{h}^\delta \phi |_{i,j} \right)^{2}
    + (\sqrt{2}-1 )\rho_{i,j}^{2}
    +\frac{1}{\delta}|\nabla_{h}\phi_{i,j}|^{2}\right)\right),   \label{discrete-energy-convex} \\
E_{e}\left(\phi,\rho\right)
=& h^{2}\sum_{i,j=1}^{N}\left(\frac{1}{8\varepsilon} (\phi_{i,j}-\frac{1}{2} )^{2}
    +\frac{\alpha}{2}\left( (\sqrt{2}-1 )\rho_{i,j}^{2}
    +\frac{1}{\delta}|\nabla_h \phi_{i,j} |^2 \right)\right) .   \label{discrete-energy-concave}    
\end{align}
In particular, a Hessian matrix could be similar formulated as the one given by~\eqref{Hessian-H-1}, for the following discrete function 
\begin{equation} 
   e_{c_2, h} (\rho, v_1, v_2, v_3, v_4 ) 
   = \frac{\alpha}{2} \Big( (\rho -  \sqrt{ \frac12 ( v_1^2 + v_2^2 + v_3^2 + v_4^2) + \delta^2} )^{2}
    + (\sqrt{2}-1 ) \rho^{2}
    +\frac{1}{2 \delta} ( v_1^2 + v_2^2 + v_3^2 + v_4^2) \Big) . 
\end{equation} 
A careful calculation reveals that, the corresponding $5 \times 5$ Hessian matrix is diagonally dominated, therefore non-negative definite. This in turn leads to the convexity of the following discrete functional: 
\begin{equation} 
   E_{c_2, h} (\rho, v_1, v_2, v_3, v_4 ) 
   = \frac{\alpha}{2} h^{2} \sum_{i,j=1}^{N}
     \left( (\rho_{i,j}- {\cal A} |\nabla_{h}^\delta \phi |_{i,j} )^{2}
    + (\sqrt{2}-1 )\rho_{i,j}^{2}
    +\frac{1}{\delta}|\nabla_{h}\phi_{i,j}|^{2}\right)  . 
\end{equation}  
The convexity analysiss for the other parts of $E_c$ and $E_e$ is more straightforward.  

We follow the idea of convexity splitting and consider the following semi-implicit, 
fully discrete scheme: given $\phi^{n}, \rho^{n} \in \mathcal{C}_{per}$, find 
$\phi^{n+1}, \rho^{n+1},\mu_{\phi}^{n+1},\mu_{\rho}^{n+1} \in \mathcal{C}_{per}$, 
such that
\begin{subequations}\label{cs_1st}
\begin{align}
\frac{\phi^{n+1}-\phi^{n}}{\Delta t} &=\mathcal{M}\Delta_{h} \mu_{\phi}^{n+1},\label{cs_1st_phi}\\
\frac{\rho^{n+1}-\rho^{n}}{\Delta t} &= \mathcal{M}\Delta_{h}\mu_{\rho}^{n+1},\label{cs_1st_rho}\\
\mu_{\phi}^{n+1} &= \delta_{\phi} E_{c}\left(\phi^{n+1},\rho^{n+1}\right)
-\delta_{\phi}E_{e}\left(\phi^{n},\rho^{n}\right),\label{cs_1st_mu1}\\
\mu_{\rho}^{n+1} &= \delta_{\rho} E_{c}\left(\phi^{n+1},\rho^{n+1}\right)
-\delta_{\rho}E_{e}\left(\phi^{n},\rho^{n}\right),\label{cs_1st_mu2}
\end{align}
\end{subequations}
where
\begin{align}
\mu_{\phi}^{n+1}=
&\frac{1}{\varepsilon} (\phi^{n+1}-\frac{1}{2} )^{3}
- (\varepsilon+\alpha+\frac{\alpha}{\delta} )\Delta_h \phi^{n+1}
+\eta^{2} \Delta_h^{2} \phi^{n+1}
+\alpha \nabla_h \cdot \left( {\cal A} ( \frac{\rho^{n+1} }{ {\cal A} |\nabla_h^\delta \phi^{n+1}|} ) \nabla_h \phi^{n+1} \right) \nonumber\\
&-\frac{1}{4\varepsilon} (\phi^{n}-\frac{1}{2} )
+\frac{\alpha}{\delta} \Delta_h \phi^{n}, \label{mu_phi}\\
\mu_{\rho}^{n+1}=
&-\xi \Delta_h \rho^{n+1}
+ \beta\left(\ln \rho^{n+1} - \ln \left(1-\rho^{n+1}\right)\right)
+\sqrt{2}\alpha\rho^{n+1}
-\alpha {\cal A} \left|\nabla_h \phi^{n+1}\right|
-\alpha(\sqrt{2}-1) \rho^{n}. \label{mu_rho}
\end{align}
Notice that $\nabla_h \cdot \left( {\cal A} ( \frac{\rho }{ {\cal A} |\nabla_h \phi|} ) \nabla_h \phi \right)$ is evaluated as follows 
\begin{align} 
  & 
  \nabla_h \cdot \left( {\cal A} ( \frac{\rho }{ {\cal A} |\nabla_h^\delta \phi|} ) \nabla_h \phi \right)  
  \nonumber 
\\
  =&  \frac{1}{h} \Big( \frac12 ( \frac{\rho_{i,j} }{  {\cal A} |\nabla_h^\delta \phi_{i,j} |}  
   + \frac{\rho_{i+1,j} }{  {\cal A} |\nabla_h^\delta \phi_{i+1,j} |} ) ( D_x \phi )_{i+\frac12, j} 
   - \frac12 ( \frac{\rho_{i,j} }{  {\cal A} |\nabla_h^\delta \phi_{i,j} |}  
   + \frac{\rho_{i-1,j} }{  {\cal A} |\nabla_h^\delta \phi_{i-1,j} |} ) ( D_x \phi )_{i-\frac12, j}  \Big)  
   \nonumber 
\\
  & 
   + \frac{1}{h} \Big( \frac12 ( \frac{\rho_{i,j} }{  {\cal A} |\nabla_h^\delta \phi_{i,j} |}  
   + \frac{\rho_{i,j+1} }{  {\cal A} |\nabla_h^\delta \phi_{i,j+1} |} ) ( D_y \phi )_{i, j+\frac12} 
   - \frac12 ( \frac{\rho_{i,j} }{  {\cal A} |\nabla_h^\delta \phi_{i,j} |}  
   + \frac{\rho_{i,j-1} }{  {\cal A} |\nabla_h^\delta \phi_{i,j-1} |} ) ( D_y \phi )_{i, j-\frac12}  \Big) .    
\end{align} 

It is observed that the finite difference scheme is a system of nonlinear equations with respect to $\phi^{n+1}$ and $\rho^{n+1}$, so we that have to solve it iteratively. 
The theoretical properties of this scheme are analyzed in the next few sections.
\subsection{The positivity-preserving property}
Of course, a point-wise bound for the grid function $\rho^{n+1}$, namely, $0<\rho_{i,j}^{n+1}< 1$, is needed to make sure the numerical scheme is well-defined. 
The main theoretical result is stated below, which assures that there exists a unique numerical solution for $\eqref{cs_1st_phi}$ and $\eqref{cs_1st_rho}$, 
so that the given bound is satisfied. 
\begin{theorem} \label{thm: positivity} 
Given $\phi^{n},\rho^{n}\in\mathcal{C}_{per}$, with $\|\rho^{n}\|_{\infty}\leq M$, for some $M\geq 0$, and $ 0 < \overline{\rho^{n}} <1$, 
there exists a unique solution $\phi^{n+1},\rho^{n+1}\in\mathcal{C}_{per}$ to $\eqref{cs_1st}$, with $0 < \rho^{n+1} <1$ at a point-wise level. 
\end{theorem}

Before the proof of the positivity-preserving property, we recall the following lemma, cited from~\cite{Chen2019Pos}.  

\begin{lemma} \label{lem: Poisson}  \cite{Chen2019Pos} 
Suppose that $\phi\in \mathring{\mathcal{C}}_{\mathrm{per}}$ and $\|\phi\|_{\infty}\leq C_1$, then we have the following estimate:
\begin{equation}\label{LEM}
\begin{aligned}
\| ( - \Delta_h )^{-1}\phi\|_{\infty}\leq C_{2} C_1 ,
\end{aligned}
\end{equation}
where $C_2$ depends only on $\Omega$.
\end{lemma}

In addition, a few more preliminary estimates are needed for the positivity-preserving analysis. The following discrete energy functional is introduced
\begin{equation} 
\begin{aligned} 
\mathcal{J}^{n}\left(\phi,\rho\right):=
&\frac{1}{2\mathcal{M}\Delta t}\|\phi-\phi^{n}\|_{-1,h}^{2}
+\frac{1}{2\mathcal{M}\Delta t}\|\rho-\rho^{n}\|_{-1,h}^{2} \\
&+\frac{1}{4\varepsilon} \| \phi-\frac{1}{2} \|_4^4 
+\frac{\varepsilon}{2}\left\|\nabla_{h}\phi\right\|_{2}^{2}
+\frac{\eta^{2}}{2}\left\|\Delta_{h}\phi\right\|_{2}^{2}
+\frac{\xi}{2}\left\|\nabla_{h}\rho\right\|_{2}^{2} \\
&+\frac{\alpha}{2}\left\{\left\|\rho- {\cal A} \left|\nabla_{h}^\delta \phi\right|\right\|_{2}^{2}
+ (\sqrt{2}-1 ) \| \rho \|_{2}^{2}
+\frac{1}{\delta}\left\|\nabla_{h}\phi\right\|_{2}^{2}\right\} \\
&+\beta\left\langle\rho,\ln\rho\right\rangle_{\Omega}
+\beta\left\langle 1-\rho,\ln\left(1-\rho\right)\right\rangle_\Omega 
 + \langle \phi , f_\phi^n \rangle_\Omega 
 + \langle \rho, f_\rho^n \rangle_\Omega , \\
f_\phi^n = & -\frac{1}{4\varepsilon} ( \phi^{n} -\frac{1}{2} ) 
-\frac{\alpha}{\delta} \Delta_{h}\phi^{n} ,  \quad 
  f_\rho^n = - (\sqrt{2}-1 )\alpha \rho^n . 
\end{aligned} 
\label{J-defi-1} 
\end{equation} 

\begin{lemma} \label{lem: J}  
Set $M_1^* = \max | f_\phi^n |$, $M_2^* = \max | f_\rho^n| \le ( \sqrt{2}-1 ) \alpha M$. We notice that $M_1^*$ may be $h$, $\varepsilon$, $\delta$ and $\phi^n$ dependent. The following inequality is available for the lower bound of $\mathcal{J}^{n} (\phi, \rho)$, for $0 < \rho <1$: 
\begin{eqnarray}  \label{lem: est-J} 
  \mathcal{J}^{n} (\phi, \rho)  \ge   
  \frac{1}{8 \varepsilon} \| \phi-\frac{1}{2} \|_4^4  - M_3^n ,  \quad 
  M_3^n = ( \frac12 ( M_1^* )^2 + \frac12 M_1^* + \frac12 \varepsilon + \beta \ln 2 + M_2^* ) | \Omega | ,  
  \label{est-J-0} 
\end{eqnarray}
\end{lemma}

\begin{proof} 
In the expansion of $\mathcal{J}^{n} (\phi, \rho)$, it is observed that all the terms are non-negative, except for the last four term. Then we get 
\begin{equation} 
\mathcal{J}^{n} (\phi,\rho) \ge 
  \frac{1}{4\varepsilon} \| \phi-\frac{1}{2} \|_4^4 
  + \langle \phi , f_\phi^n \rangle_\Omega   
 +\beta ( \left\langle\rho,\ln\rho\right\rangle_{\Omega}
+ \left\langle 1-\rho,\ln\left(1-\rho\right)\right\rangle_\Omega ) 
 + \langle \rho, f_\rho^n \rangle_\Omega . \label{est-J-1} 
\end{equation} 
For the third and fourth terms on the right hand side of~\eqref{est-J-1}, the following point-wise lower bound is available: 
\begin{equation} 
\begin{aligned} 
  &
  \rho \ln \rho + ( 1-\rho) \ln (1-\rho)  \ge 2 \cdot \frac12 \ln \frac12 
  = - \ln 2 ,  \quad \mbox{for $0 < \rho < 1$} , 
\\
  & 
  \mbox{which in turn leads to} \quad 
  \left\langle\rho,\ln\rho\right\rangle_{\Omega}
+ \left\langle 1-\rho,\ln\left(1-\rho\right)\right\rangle_\Omega  
\ge - | \Omega | \ln 2 . 
\end{aligned} 
  \label{est-J-2} 
\end{equation} 
For the last term on the right hand side of~\eqref{est-J-1}, the following lower bound is valid, based on the fact that $0 < \rho < 1$: 
\begin{equation} 
   \langle \rho, f_\rho^n \rangle_\Omega  \ge - \| f_\rho^n \|_\infty \cdot \| {\bf 1} \|_1 
   \ge   - M_2^*  | \Omega |  .  \label{est-J-3} 
\end{equation}    
For the first two terms on the right hand side of~\eqref{est-J-1}, we begin with the following estimate: 
\begin{eqnarray} 
\begin{aligned} 
    \langle \phi , f_\phi^n \rangle_\Omega 
    =&  \langle \phi - \frac12 , f_\phi^n \rangle_\Omega  
    +   \frac12 \langle {\bf 1}  , f_\phi^n \rangle_\Omega   
    \ge - \frac12 ( \| \phi - \frac12 \|_2^2 + \| f_\phi^n \|_2^2 ) 
    - \frac12 M_1^* | \Omega |  
\\
  \ge& 
  - \frac12 \| \phi - \frac12 \|_2^2 - \frac12 ( M_1^* )^2 | \Omega |  
    - \frac12 M_1^* | \Omega |   .   
\end{aligned} 
\label{est-J-4-1} 
\end{eqnarray}   
Meanwhile, the following lower bound is a direct application of Cauchy inequality: 
\begin{equation} 
     \frac{1}{8\varepsilon} \| \phi-\frac{1}{2} \|_4^4  
     -   \frac12 \| \phi - \frac12 \|_2^2 \ge -\frac12 \varepsilon | \Omega | , 
     \label{est-J-4-2}      
\end{equation} 
and its combination with~\eqref{est-J-4-1} yields  
\begin{eqnarray} 
    \frac{1}{8\varepsilon} \| \phi-\frac{1}{2} \|_4^4  +
    \langle \phi , f_\phi^n \rangle_\Omega     
  \ge - ( \frac12 ( M_1^* )^2 + \frac12 M_1^*  + \frac12 \varepsilon ) | \Omega |   .   
\label{est-J-4-3} 
\end{eqnarray}          
 Finally, a substitution of~\eqref{est-J-2}, \eqref{est-J-3} and \eqref{est-J-4-3} into \eqref{est-J-1} results in~\eqref{est-J-0}. This completes the proof of Lemma~\ref{lem: J}.     
\end{proof} 

Now we proceed into the proof of Theorem~\ref{thm: positivity}. 

\begin{proof}
We denote $M_0^n = \mathcal{J}^{n} (\phi^n, \rho^n)$, a fixed constant with given $(\phi^n, \rho^n)$. The numerical solution of $\eqref{cs_1st}$ is a minimizer of the discrete energy functional $\mathcal{J}^{n}\left(\phi,\rho\right)$ (defined in~\eqref{J-defi-1}), over the admissible set
\begin{equation} 
\begin{aligned} 
  & 
A_{h}:=\left\{\phi,\rho\in\mathcal{C}_{per}| -A^* < \phi < A^*, \, 0 < \rho < 1,\left\langle\phi-\bar{\phi}_{0}\right\rangle_{\Omega}=0,\left\langle\rho-\bar{\rho}_{0}\right\rangle_{\Omega}=0\right\} \subset \mathbb{R}^{N^2} \times \mathbb{R}^{N^2} , 
\\
  & 
  A^* := \max \Big( ( 8  \varepsilon (M_0^n + M_3 ^n ) )^\frac14 h^{-\frac12} , \| \phi^n \|_\infty \Big) + 1 . 
\end{aligned} 
 \label{defi-A-delta}   
\end{equation}
We can observe that $\mathcal{J}^{n}$ is a strictly convex function over this set.

To facilitate the analysis below, we transform the minimization problem into an equivalent one. Consider the functional
\begin{align}
\mathcal{F}^{n}\left(\phi,\rho\right)
:=&\mathcal{J}^{n}\left(\phi+\bar{\phi}_{0},\rho+\bar{\rho}_{0}\right)\nonumber\\
=&\frac{1}{2\mathcal{M}\Delta t}\|\phi+\bar{\phi}_{0}-\phi^{n}\|_{-1,h}^{2}
+\frac{1}{2\mathcal{M}\Delta t}\|\rho+\bar{\rho}_{0}-\rho^{n}\|_{-1,h}^{2}\nonumber\\
&+\frac{1}{4\varepsilon} \| (\phi+\bar{\phi}_{0}-\frac{1}{2} )^{2} \|_{2}^{2}
+\frac{\varepsilon}{2}\left\|\nabla_{h}\phi\right\|_{2}^{2}
+\frac{\eta^{2}}{2}\left\|\Delta_{h}\phi\right\|_{2}^{2}
+\frac{\xi}{2}\left\|\nabla_{h}\rho\right\|_{2}^{2}\nonumber\\
&+\frac{\alpha}{2}\left\{\left\|\rho+\bar{\rho}_{0}
- {\cal A} |\nabla_{h}^\delta \phi| \right\|_{2}^{2}
+ (\sqrt{2}-1 )\left\|\rho+\bar{\rho}_{0}\right\|_{2}^{2}
+\frac{1}{\delta}\left\|\nabla_{h}\phi\right\|_{2}^{2}\right\}\nonumber\\
&+\beta\left\langle\rho+\bar{\rho}_{0},\ln\left(\rho+\bar{\rho}_{0}\right)\right\rangle_{\Omega}
+\beta\left\langle 1-\rho-\bar{\rho}_{0},\ln\left(1-\rho-\bar{\rho}_{0}\right)\right\rangle_{\Omega}\nonumber\\
& + \langle\phi+\bar{\phi}_{0}, f_\phi^n \rangle_\Omega 
 + \langle\rho+\bar{\rho}_{0}, f_\rho^n \rangle_{\Omega},
\end{align}
defined on the set
\begin{equation*}
\mathring{A}_{h}:=
\left\{\left(\phi,\rho\right)\in\mathring{\mathcal{C}}_{per}\times\mathring{\mathcal{C}}_{per} | -A^* - \bar{\phi}_0 < \phi < A^*- \bar{\phi}_0 , -\bar{\rho}_{0} < \rho < 1-\bar{\rho}_{0}\right\}
\subset\mathbb{R}^{N^{2}}\times\mathbb{R}^{N^{2}}.
\end{equation*}
If $(\phi, \rho) \in\mathring{A}_{h}$ minimizes $\mathcal{F}^{n}$, then $(\varphi, \varrho) :=( \phi+\bar{\phi}_{0}, \rho + \bar{\rho}_0) \in A_{h}$ minimizes $\mathcal{J}^{n}$ and $vice~versa$. 
Next, we prove that there exists a minimizer of $\mathcal{F}^{n}$ over the domain $\mathring{A}_{h}$. Consider the following closed domain: for $\delta_0 \in\left(0,\frac{1}{2}\right)$
\begin{equation*}  
\begin{aligned} 
\mathring{A}_{h,\delta_0}:= & \left\{\left(\phi,\rho\right)\in\mathring{\mathcal{C}}_{per}\times\mathring{\mathcal{C}}_{per} | -A^* - \bar{\phi}_0 + \delta_0 \le \phi \le A^*- \bar{\phi}_0 - \delta_0 , \, \delta_0 -\bar{\rho}_{0}\leq \rho \leq 1-\delta_0 -\bar{\rho}_{0}\right\}  
\\
  & 
  \subset\mathbb{R}^{N^{2}}\times\mathbb{R}^{N^{2}}. 
\end{aligned} 
\end{equation*}
Since $\mathring{A}_{h,\delta_0}$ is a bounded, compact, and convex set in the subspace $\mathring{\mathcal{C}}_{per}$, 
there exists a (not necessarily unique) minimizer of $\mathcal{F}^{n}$ over $\mathring{A}_{h,\delta_0}$. The key point of the positivity analysis is that, such a minimizer could not occur on the boundary of $\mathring{A}_{h,\delta_0}$, if $\delta_0$ is sufficiently small. To be more explicit, by the boundary of $\mathring{A}_{h,\delta_0}$, 
we mean the locus of points $\left(\phi,\psi\right)\in\mathring{A}_{h,\delta_0}$ such that $\psi+\bar{\rho}_{0} =\delta_0$ or $1-\delta_0$, or $\phi+\bar{\phi}_{0} =-A^* + \delta_0$ or $A^* -\delta_0$, precisely.

To get a contradiction, suppose that the minimizer of $\mathcal{F}^{n}$, call it $\left(\phi^{\star},\rho^{\star}\right)$, occurs at a boundary point of $\mathring{A}_{h,\delta_0}$. 
There is at least one grid point $\vec{\alpha}_{0}=\left(i_{0},j_{0}\right)$ such that $\rho_{\vec{\alpha}_{0}}^{\star}+\bar{\rho}_{0} = \delta_0$ or $1-\delta_0$, or $\phi_{\vec{\alpha}_{0}}^{\star} +\bar{\phi}_{0} =-A^* + \delta_0$ or $A^* -\delta_0$. Let us assume that $\rho_{\vec{\alpha}_{0}}^{\star}+\bar{\rho}_{0}=\delta_0$, 
and denote by $\vec{\alpha}_{1}=\left(i_{1},j_{1}\right)$ as the grid point at which $\rho^{\star}$ achieves its maximum. By the fact that $\bar{\rho^{\star}}=0$, it is obvious that $\rho_{\vec{\alpha}_{1}}^\star \ge 0$ and 
\begin{equation*}
1-\delta_0 \geq \rho_{\vec{\alpha}_{1}}^{\star}+\bar{\rho}_{0}\geq \bar{\rho}_{0}.
\end{equation*}
Since 
$\mathcal{F}^{n}$ is smooth over $\mathring{A}_{h,\delta}$, for all 
$\left(\varphi,\psi\right)\in\mathring{\mathcal{C}}_{per}$, the directional derivative is 
\begin{align}
&\mathrm{d}_{s}\mathcal{F}^{n}\left(\phi^{\star}+s\varphi,\rho^{\star}+s\psi\right)|_{s=0}\nonumber\\
=&\frac{1}{\mathcal{M}\Delta t}\left\langle\left(-\Delta_{h}\right)^{-1}
\left(\phi^{\star}+\bar{\phi}_{0}-\phi^{n}\right),\varphi\right\rangle_{\Omega}
+\frac{1}{\mathcal{M}\Delta t}\left\langle\left(-\Delta_{h}\right)^{-1}
\left(\rho^{\star}+\bar{\rho}_{0}-\rho^{n}\right),\psi\right\rangle_{\Omega}\nonumber\\
&+\frac{1}{2\varepsilon} \langle\phi^{\star}+\bar{\phi}_{0}-\frac{1}{2},\varphi \rangle_{\Omega}
-\varepsilon\left\langle\Delta_{h}\phi^{\star},\varphi\right\rangle_{\Omega}
+\eta^{2}\left\langle\Delta_{h}^{2}\phi^{\star},\varphi\right\rangle_{\Omega}
-\xi\left\langle\Delta_{h}\rho^{\star},\psi\right\rangle_{\Omega}\nonumber\\
&+\beta\left\langle\ln\left(\rho^{\star}+\bar{\rho}_{0}\right)
-\ln\left(1-\rho^{\star}-\bar{\rho}_{0}\right),\psi\right\rangle_{\Omega}
-\alpha\left\langle \nabla_{h}\cdot \Big( {\cal A} ( \frac{\rho^{\star}+\bar{\rho}_{0} }{ {\cal A} |\nabla_{h}^\delta \phi^{\star}|} ) \nabla_h \phi^{\star} \Big) ,\varphi\right\rangle_{\Omega}\nonumber\\
&-\alpha\left\langle {\cal A} |\nabla_{h}^\delta \phi^{\star}|,\psi\right\rangle_{\Omega}
+\alpha (\sqrt{2}-1 )\left\langle\rho^{\star}+\bar{\rho}_{0},\psi\right\rangle_{\Omega}
-\frac{\alpha}{\delta}\left\langle\Delta_{h}\phi^{\star},\varphi\right\rangle_{\Omega} 
+ \langle f_\phi^n ,\varphi \rangle_{\Omega} 
+ \langle f_\rho^n, \psi \rangle_{\Omega} . 
\end{align}
Here, we take the direction $\varphi,\psi\in\mathring{C}_{per}$, such that
\begin{equation*}
\varphi = 0,\quad\psi=\delta_{i,i_{0}}\delta_{j,j_{0}}-\delta_{i,i_{1}}\delta_{j,j_{1}}.
\end{equation*}
Then the derivative may be expressed as
\begin{align}\label{estimate_direction}
&\frac{1}{h^{2}}\mathrm{d}_{s}\mathcal{F}^{n}\left(\phi^{\star},\rho^{\star}+s\psi\right)|_{s=0}\nonumber\\
=&\frac{1}{\mathcal{M}\Delta t}\left(-\Delta_{h}\right)^{-1}
\left(\rho^{\star}+\bar{\rho}_{0}-\rho^{n}\right)_{\vec{\alpha}_{0}}
-\frac{1}{\mathcal{M}\Delta t}\left(-\Delta_{h}\right)^{-1}
\left(\rho^{\star}+\bar{\rho}_{0}-\rho^{n}\right)_{\vec{\alpha}_{1}}
-\xi\left(\Delta_{h}\rho_{\vec{\alpha}_{0}}^{\star}-\Delta_{h}\rho_{\vec{\alpha}_{1}}^{\star}\right)\nonumber\\
&+\beta\ln\left(\rho_{\vec{\alpha}_{0}}^{\star}+\bar{\rho}_{0}\right)
-\beta\ln\left(1-\rho_{\vec{\alpha}_{0}}^{\star}-\bar{\rho}_{0}\right)
-\beta\ln\left(\rho_{\vec{\alpha}_{1}}^{\star}+\bar{\rho}_{0}\right)
+\beta\ln\left(1-\rho_{\vec{\alpha}_{1}}^{\star}-\bar{\rho}_{0}\right)\nonumber\\
&-\alpha\left( {\cal A} |\nabla_{h}^\delta \phi_{\vec{\alpha}_{0}}^{\star}| 
- {\cal A} |\nabla_{h}^\delta \phi_{\vec{\alpha}_{1}}^{\star}|\right)
+\alpha (\sqrt{2}-1 )\left(\rho_{\vec{\alpha}_{0}}^{\star}-\rho_{\vec{\alpha}_{1}}^{\star}\right)
 + ( f_\rho^n )_{\vec{\alpha}_{0}} - (f_\rho^n )_{\vec{\alpha}_{1}} . 
\end{align}
For simplicity, now let us write $\varrho^{\star}:=\rho^{\star}+\bar{\rho}_{0}$. Since 
$\varrho_{\vec{\alpha}_{0}}^{\star}=\delta_0$ and $\varrho_{\vec{\alpha}_{1}}^{\star}\geq\bar{\rho}_{0}$, we have
\begin{equation}\label{estimate_log}
\ln\left(\varrho_{\vec{\alpha}_{0}}^{\star}\right)
-\ln\left(1-\varrho_{\vec{\alpha}_{0}}^{\star}\right)
-\ln\left(\varrho_{\vec{\alpha}_{1}}^{\star}\right)
+\ln\left(1-\varrho_{\vec{\alpha}_{1}}^{\star}\right) 
\le \ln \frac{\delta_0}{1-\delta_0}
-\ln\frac{\bar{\rho}_{0}}{1-\bar{\rho}_{0}}.
\end{equation}
Since $\varrho^{\star}$ takes a minimum at the grid point $\vec{\alpha}_{0}$, with 
$\varrho_{\vec{\alpha}_{0}}^{\star}=\delta\leq\varrho_{i,j}^{\star}$, for any $\left(i,j\right)$, 
and a maximum at the grid point $\vec{\alpha}_{1}$, with 
$\varrho_{\vec{\alpha}_{1}}^{\star}\geq\varrho_{i,j}^{\star}$, for any $\left(i,j\right)$,
\begin{equation}\label{estimate_star}
\Delta_{h}\rho_{\vec{\alpha}_{0}}^{\star}\geq 0,\quad
\Delta_{h}\rho_{\vec{\alpha}_{1}}^{\star}\leq 0,\quad
\rho_{\vec{\alpha}_{0}}^{\star}-\rho_{\vec{\alpha}_{1}}^{\star}\leq 0.
\end{equation}
For the numerical solution $\rho^{n}$ at the previous time step, the a priori assumption 
$\left\|\rho^{n}\right\|_{\infty}\leq M$ indicates that
\begin{equation}\label{estimate_n}
-2M\leq\rho_{\vec{\alpha}_{0}}^{n}-\rho_{\vec{\alpha}_{1}}^{n}\leq 2M.
\end{equation}
According to Lemma~\ref{lem: Poisson}, we obtain
\begin{equation}\label{estimate_lap}
- 4M C_2 \leq\left(-\Delta_{h}\right)^{-1}\left(\varrho^{\star}-\rho^{n}\right)_{\vec{\alpha}_{0}}
-\left(-\Delta_{h}\right)^{-1}\left(\varrho^{\star}-\rho^{n}\right)_{\vec{\alpha}_{1}}\leq 4M C_2.
\end{equation}
Denote $C_{3}= \max\left\{ {\cal A} |\nabla_{h}^\delta \phi_{\vec{\alpha}_{0}}^{\star}|, {\cal A} |\nabla_{h}^\delta \phi_{\vec{\alpha}_{1}}^{\star}|\right\}$. Based on the fact that $-A^* - \bar{\phi}_0 < \phi^* < A^*- \bar{\phi}_0$ at a point-wise level, we conclude that 
\begin{equation} 
    {\cal A} |\nabla_{h}^\delta \phi_{\vec{\alpha}_{0}}^{\star}|,\, \, 
    {\cal A} |\nabla_{h}^\delta \phi_{\vec{\alpha}_{1}}^{\star} |  
    \le \frac{2 A^*}{h} + 1 ,   \quad 
    \mbox{so that} \, \, \, C_3 \le \frac{2 A^*}{h} + 1 . 
\end{equation}     
Then we have 
\begin{equation}\label{estimate_phi-1}
-\alpha C_{3}\leq
-\alpha {\cal A} \left(|\nabla_{h}^\delta \phi_{\vec{\alpha}_{0}}^{\star}| 
- |\nabla_{h}^\delta \phi_{\vec{\alpha}_{1}}^{\star}|\right) \le \alpha C_{3}.
\end{equation}
Consequently, a substitution of $\eqref{estimate_log},\eqref{estimate_star},\eqref{estimate_n},\eqref{estimate_lap},\eqref{estimate_phi-1}$ into $\eqref{estimate_direction}$ yields 
the following bound on the directional derivative:
\begin{equation*}
\frac{1}{h^{2}}\mathrm{d}_{s}\mathcal{F}^{n}\left(\phi^{\star},\rho^{\star}+s\psi\right)|_{s=0}
\leq \beta \ln \frac{\delta_0}{1-\delta_0} - \beta\ln\frac{\bar{\rho}_{0}}{1-\bar{\rho}_{0}}
+ 4M C_2 (\mathcal{M}\Delta t)^{-1}+2\alpha C_{3}+ 2 M_2 .
\end{equation*}
We denote $D_0 = 4M C_2 (\mathcal{M}\Delta t)^{-1}+2\alpha C_{3}+ 2 M_2$. Notice that $D_0$ is a constant for fixed $\Delta t$ and $h$, though it becomes singular as $\Delta t\to 0$ and $h\to 0$. On the other hand, for any fixed $\Delta t$ and $h$, we may choose $\delta_0 \in\left(0,1/2\right)$ sufficiently small so that
\begin{equation}\label{estimate_C}
\beta \ln \frac{\delta_0}{1-\delta_0} - \beta\ln\frac{\bar{\rho}_{0}}{1-\bar{\rho}_{0}} 
+ D_0 < 0.
\end{equation}
This in turn leads to the following inequality, provided $\delta_0$ satisfies $\eqref{estimate_C}$,
\begin{equation}
\frac{1}{h^{2}}\mathrm{d}_{s}\mathcal{F}^{n}\left(\phi^{\star},\rho^{\star}+s\psi\right)|_{s=0} < 0.
\end{equation}
As before, this contradicts the assumption that $\mathcal{F}^{n}$ has a minimum at $\left(\phi^{\star},\rho^{\star}\right)$, since the directional derivative is negative in a direction pointing into the interior of $\mathring{A}_{h,\delta_0}$.

Using very similar arguments, we can also prove that the global minimum of $\mathcal{F}^{n}$ over $\mathring{A}_{h,\delta_0}$ could not occur at a boundary point $\left(\phi^{\star},\rho^{\star}\right)$ such that $\rho_{\vec{\alpha}_{0}}^{\star}+\bar{\rho}_{0}=1-\delta_0$, for some $\vec{\alpha}_{0}$, so that the grid function $\rho^{\star}$ has a global maximum at $\vec{\alpha}_{0}$. The details are left to interested readers. 

Moreover, if the global minimum of $\mathcal{F}^{n}$ over $\mathring{A}_{h,\delta_0}$ could occurs at a boundary point $\left(\phi^{\star},\rho^{\star}\right)$ such that $\varphi_{\vec{\alpha}_{0}}^* = \phi_{\vec{\alpha}_{0}}^{\star}+\bar{\phi}_{0}=A^*-\delta_0$. In turn, we apply Lemma~\ref{lem: J} and obtain 
\begin{eqnarray}  
\begin{aligned} 
  \mathcal{F}^{n} (\phi^* , \rho^*) = & \mathcal{J}^{n} (\varphi^*, \varrho^*)  
  \ge  \frac{1}{8 \varepsilon} \| \varphi^* -\frac{1}{2} \|_4^4  - M_3^n  
  \ge  \frac{1}{8 \varepsilon} h^2 \cdot ( \varphi_{\vec{\alpha}_{0}}^* - \frac12 )^4 
    - M_3^n   
\\
  \ge & 
     \frac{1}{8 \varepsilon} h^2 ( A^* - \delta_0 - \frac12 )^4 - M_3^n     
     >  \frac{1}{8 \varepsilon} h^2 ( A^* - 1 )^4 - M_3^n     
\\
  \ge & 
     \frac{1}{8 \varepsilon} h^2 \cdot 
     ( 8  \varepsilon (M_0^n + M_3 ^n ) ) h^{-2} - M_3^n  
     = M_0^n = \mathcal{J}^{n} ( \phi^n, \rho^n)  , 
\end{aligned} 
\label{positivity-2-1} 
\end{eqnarray}
in which the definition of $A^*$ (in~\eqref{defi-A-delta}) has been recalled. This contradicts the assumption that $\mathcal{F}^{n}$ has a minimum at $\left(\phi^{\star},\rho^{\star}\right)$.  

Using similar arguments, a minimization point cannot occur at a boundary point $\left(\phi^{\star},\rho^{\star}\right)$ such that $\varphi_{\vec{\alpha}_{0}}^* = \phi_{\vec{\alpha}_{0}}^{\star}+\bar{\phi}_{0}=-A^*+ \delta_0$. The details are left to interested readers. 

A combination of above four facts have indicated that, the global minimum of $\mathcal{F}^{n}$ over $\mathring{A}_{h,\delta_0}$ could only possibly occur at interior point $\left(\phi,\rho\right)\in\left(\mathring{A}_{h,\delta_0}\right)^{\circ} \subset\left(\mathring{A}_{h}\right)^{\circ}$. We conclude that there must be a solution $\left(\phi,\rho\right)\in A_{h}$ that minimizes $\mathcal{J}^{n}$ over $A_{h}$, which is equivalent to the numerical solution of $\eqref{cs_1st}$. The existence of the numerical solution is established. 

In addition, since $\mathcal{J}^{n}$ is strictly convex function over $A_{h}$, the uniqueness analysis for this numerical solution is straightforward. The proof of Theorem~\ref{thm: positivity} is completed.
\end{proof}
\subsection{Unconditional energy stability}
\begin{theorem}[Energy stability] \label{thm: energy stability}  
For $n\geq 1$, the numerical scheme $\eqref{cs_1st}$ is unconditionally energy stable, i.e. 
\begin{equation*}
E\left(\phi^{n+1},\rho^{n+1}\right)\leq E\left(\phi^{n},\rho^{n}\right).
\end{equation*}
\end{theorem}
\begin{proof}

Let $\mathcal{L}=-\Delta_{h}$. Due to the mass conservation, $\mathcal{L}^{-1}\left(\phi^{n+1}-\phi^{n}\right)$ and $\mathcal{L}^{-1}\left(\rho^{n+1}-\rho^{n}\right)$ 
are well-defined. Taking a discrete inner product with $\eqref{cs_1st_phi}$,  $\eqref{cs_1st_rho}$, $\eqref{cs_1st_mu1}$, $\eqref{cs_1st_mu2}$ by $\mathcal{L}^{-1}\left(\phi^{n+1}-\phi^{n}\right)$, $\mathcal{L}^{-1}\left(\rho^{n+1}-\rho^{n}\right)$, $\phi^{n+1}-\phi^{n}$ and $\rho^{n+1}-\rho^{n}$, respectively, yields the following estimate 
\begin{align}
0 =& \frac{1}{\mathcal{M}\Delta t}\left\langle\phi^{n+1}-\phi^{n}, 
\mathcal{L}^{-1}\left(\phi^{n+1}-\phi^{n}\right)\right\rangle_{\Omega}
+\frac{1}{\mathcal{M}\Delta t}\left\langle\rho^{n+1}-\rho^{n}, 
\mathcal{L}^{-1}\left(\rho^{n+1}-\rho^{n}\right)\right\rangle_{\Omega}\nonumber\\
   &+\left\langle\delta_{\phi} E_{c}\left(\phi^{n+1},\rho^{n+1}\right)
   -\delta_{\phi}E_{e}\left(\phi^{n},\rho^{n}\right),\phi^{n+1}-\phi^{n}\right\rangle_{\Omega}\nonumber\\
   &+\left\langle\delta_{\rho} E_{c}\left(\phi^{n+1},\rho^{n+1}\right)
   -\delta_{\rho}E_{e}\left(\phi^{n},\rho^{n}\right),\rho^{n+1}-\rho^{n}\right\rangle_{\Omega}\nonumber\\
\geq & E\left(\phi^{n+1},\rho^{n+1}\right)-E\left(\phi^{n},\rho^{n}\right).\label{1st_energy_product}
\end{align}
Hence that
\begin{equation*}
E\left(\phi^{n+1},\rho^{n+1}\right)\leq E\left(\phi^{n},\rho^{n}\right).
\end{equation*}
This completes the proof.
\end{proof}

\subsection{Optimal rate convergence analysis}
Let $\Phi$ and $\Psi$ be the exact solution for the binary fluid-surfactant system $\eqref{bfs}$.
With initial data with sufficient regularity, we could assume that the exact solution 
has regularity of class $\mathcal{R}_{1}$ and $\mathcal{R}_{2}$:
\begin{subequations}\label{regular}
\begin{align}
&\Phi\in \mathcal{R}_{1}:=H^{2}\left(0,T;C_{per}\left(\Omega\right)\right)
\cap L^{\infty}\left(0,T;C_{per}^{8}\left(\Omega\right)\right),\\
&\Psi\in \mathcal{R}_{2}:=H^{2}\left(0,T;C_{per}\left(\Omega\right)\right)
\cap L^{\infty}\left(0,T;C_{per}^{6}\left(\Omega\right)\right).
\end{align}
\end{subequations}
Define $\Phi_{N}\left(\cdot,t\right):=\mathcal{P}_{N}\Phi\left(\cdot,t\right)$, and $\Psi_{N}\left(\cdot,t\right):=\mathcal{P}_{N}\Psi\left(\cdot,t\right)$, 
the spatial Fourier projection of the exact solutions into $\mathcal{B}^{K}$, the space of trigonometric polynomials of degree to and including $K$ (with $N=2K+1$). 
The following projection approximation is standard: if $\Phi \in L^{\infty}\left(0,T;H_{per}^{l}\left(\Omega\right)\right)$ 
for some $l\in \mathbb{N}$,
\begin{equation}\label{err_for_Phi}
\left\|\Phi_{N}-\Phi\right\|_{L^{\infty}\left(0,T;H^{k}\left(\Omega\right)\right)}
\leq Ch^{l-k}\left\|\Phi\right\|_{L^{\infty}\left(0,T;H^{l}\left(\Omega\right)\right)},\quad 0\leq k\leq l.
\end{equation}
By $\Phi_{N}^{m},\Phi^{m}$, 
we denote $\Phi_{N}\left(\cdot,t_{m}\right)$ and $\Phi\left(\cdot,t_{m}\right)$, respectively, 
with $t_{m}=m\cdot\Delta t$. Since $\Phi_{N}\in\mathcal{B}^{m}$, 
the mass conservative property is available at the discrete level:
\begin{equation*}
\overline{\Phi_{N}^{m}}
=\frac{1}{|\Omega|}\int_{\Omega}\Phi_{N}\left(\cdot,t_{m}\right)\mathrm{d}\bm{x}
=\frac{1}{|\Omega|}\int_{\Omega}\Phi_{N}\left(\cdot,t_{m+1}\right)\mathrm{d}\bm{x}
=\overline{\Phi_{N}^{m+1}}, \quad m\in\mathbb{N}.
\end{equation*}
We have a similar result about $\Psi$. On the other hand, the solution of $\eqref{bfs}$ is also mass conservative at the discrete level:
\begin{equation}
\overline{\phi^{m}}=\overline{\phi^{m+1}}, \quad
\overline{\rho^{m}}=\overline{\rho^{m+1}}, \quad m\in\mathbb{N}.
\end{equation}
As indicated before, we use the mass conservative interpolation for the initial data: 
$\phi^{0}=\mathcal{P}_{h}\Phi_{N}\left(\cdot,t=0\right)$ and 
$\rho^{0}=\mathcal{P}_{h}\Psi_{N}\left(\cdot,t=0\right)$, that is 
\begin{equation}
\phi^{0}_{i,j} = \mathcal{P}_{h} ( \Phi_{N} ( \cdot, t=0) )_{i,j} 
:=\Phi_{N}\left(x_{i}, y_{j},t=0\right),\quad
\rho^{0}_{i,j} = \mathcal{P}_{h} ( \Psi_{N} ( \cdot, t=0) )_{i,j} 
 :=\Psi_{N}\left( x_{i}, y_{j},t=0\right).  \label{interpolation-1} 
\end{equation}
The error grid function is defined as 
\begin{equation}
\tilde{\phi}^{m}:=\mathcal{P}_{h}\Phi_{N}^{m}-\phi^{m}, \quad
\tilde{\rho}^{m}:=\mathcal{P}_{h}\Psi_{N}^{m}-\rho^{m}, \quad m\in\mathbb{N} , 
\label{error function-1} 
\end{equation}
in which a similar interpolation formula could be applied to ${\cal P}_h$ as in~\eqref{interpolation-1}. Therefore, it follows that $\overline{\tilde{\phi}^{m}}=0$ and $\overline{\tilde{\rho}^{m}}=0$, for any $m\in\mathbb{N}$, so that the discrete norm $\left\|\cdot\right\|_{-1,h}$ is well defined for the numerical error grid function.
\begin{theorem} \label{thm: convergence} 
Given initial data $\Phi(\cdot,t=0)\in C_{per}^{8}(\Omega)$ and $\Psi(\cdot,t=0)\in C_{per}^{6}(\Omega)$, 
suppose the exact solutin for binary fluid-surfactant system $\eqref{bfs}$ is of regularity class $\mathcal{R}=\mathcal{R}_{1}\times\mathcal{R}_{2}$.
Then, provided that $\Delta t$ is sufficiently small, for all positive integers n, such that $t_n\leq T$, we have
\begin{equation}
\|\tilde{\phi}^{n+1}\|_{-1,h} + \|\tilde{\rho}^{n+1}\|_{-1,h} 
+\left(\mathcal{M}\Delta t\sum_{k=0}^{n+1}
\left(\eta^{2}\|\Delta_{h}\tilde{\phi}^{k}\|_{2}^{2}
+\varepsilon\|\nabla_{h}\tilde{\phi}^{k}\|_{2}^{2}
+\xi\|\nabla_{h}\tilde{\rho}^{k}\|_{2}^{2}\right)\right)^{\frac{1}{2}}
\leq C\left(\Delta t+h^{2}\right) , 
\end{equation}
where $C>0$ is independent of n, $\Delta t$, and h.
\end{theorem}

\begin{proof}
A carefully consistency analysis indicates the following truncation error estimate:
\begin{subequations}\label{estimate_1st}
\begin{align}
\frac{\Phi^{n+1}_{N}-\Phi^{n}_{N}}{\Delta t}
=&\mathcal{M}\Delta_{h}\left(\eta^{2}\Delta_{h}^{2}\Phi^{n+1}_{N}
-\varepsilon\Delta_{h}\Phi^{n+1}_{N}
+\frac{1}{\varepsilon} (\Phi^{n+1}_{N}-\frac{1}{2} )^{3}
+\alpha\nabla_{h} \cdot \left( {\cal A} ( \frac{\Psi^{n+1}_{N}}
{{\cal A} |\nabla_{h}^\delta \Phi^{n+1}_{N}|} ) \nabla_{h}\Phi^{n+1}_{N} \right)\right.\nonumber\\
&\left. -\alpha\left(1+\frac{1}{\delta}\right)\Delta_{h}\Phi^{n+1}_{N}
-\frac{1}{4\varepsilon} (\Phi^{n}_{N}-\frac{1}{2} )
-\frac{\alpha}{\delta}\Delta_{h}\Phi^{n}_{N}\right)+\tau_{\phi},
\label{estimate_phi} \\
\frac{\Psi^{n+1}_{N}-\Psi^{n}_{N}}{\Delta t}=&\mathcal{M}\Delta_{h}
\left(-\xi\Delta_{h}\Psi_{N}^{n+1}
+\beta H^{\prime}\left(\Psi^{n+1}_{N}\right)
+\sqrt{2}\alpha \Psi^{n+1}_{N} - \alpha {\cal A} | \nabla_h^\delta \Phi^{n+1}_{N} |
-\alpha\left(\sqrt{2}-1\right)\Psi^{n}_{N}\right)+\tau_{\rho},
\label{estimate_rho}
\end{align}
\end{subequations}
with $\left\|\tau^{n}\right\|_{-1,h}\leq C\left(\Delta t+h^{2}\right)$. Observe that we have dropped the operator $\mathcal{P}_{h}$, which should appear in front of $\Phi_{N}$, 
for simplicity of presentation. 

Subtracting the numerical scheme $\eqref{cs_1st}$ from $\eqref{estimate_1st}$ gives
\begin{subequations}\label{error_1st}
\begin{align}
\frac{\tilde{\phi}^{n+1}_{N}-\tilde{\phi}^{n}_{N}}{\Delta t}
=&\mathcal{M}\Delta_{h}\left(\eta^{2}\Delta_{h}^{2}\tilde{\phi}^{n+1}
-\varepsilon\Delta_{h}\tilde{\phi}^{n+1}
+\frac{1}{\varepsilon} (\Phi^{n+1}_{N}-\frac{1}{2} )^{3}
-\frac{1}{\varepsilon} (\phi^{n+1}-\frac{1}{2} )^{3}\right.\nonumber\\
&\left.+\alpha \nabla_{h} \cdot \left( {\cal A} ( \frac{\Psi^{n+1}_{N}}
{{\cal A} |\nabla_{h}^\delta \Phi^{n+1}_{N}|} ) \nabla_{h}\Phi^{n+1}_{N} \right) 
-\alpha \nabla_{h} \cdot \left( {\cal A} ( \frac{\rho^{n+1}}
{{\cal A} |\nabla_{h}^\delta \phi^{n+1} |} ) \nabla_{h} \phi^{n+1} \right) \right.\nonumber\\
&\left.-\alpha (1+\frac{1}{\delta} )\Delta_{h} \tilde{\phi}^{n+1} 
-\frac{1}{4\varepsilon}\tilde{\phi}^{n}
-\frac{\alpha}{\delta}\Delta_{h}\tilde{\phi}^{n}\right)+\tau_{\phi},
\label{error_phi}\\
\frac{\tilde{\rho}^{n+1}-\tilde{\rho}^{n}}{\Delta t}=
&\mathcal{M}\Delta_{h}
\left(-\xi\Delta_{h}\tilde{\rho}^{n+1}
+\beta H^{\prime}\left(\Psi^{n+1}_{N}\right)
-\beta H^{\prime}\left(\rho^{n+1}\right)
+\sqrt{2}\alpha \tilde{\rho}^{n+1}\right.\nonumber\\
&\left.-\alpha {\cal A} | \nabla_h^\delta \Phi^{n+1}_{N} | 
 + \alpha {\cal A} | \nabla_h^\delta \phi^{n+1}|
-\alpha (\sqrt{2}-1 )\tilde{\rho}^{n}\right)+\tau_{\rho}.
\label{error_rho}
\end{align}
\end{subequations}
Since the numerical error function has zero-mean, we see that both $\left(-\Delta_{h}\right)^{-1}\tilde{\phi}^{m}$ and $\left(-\Delta_{h}\right)^{-1}\tilde{\rho}^{m}$ are well-defined, 
for any $k\geq 0$. Taking a discrete inner product with $\eqref{estimate_phi}$ and $\eqref{estimate_rho}$ by  $2\left(-\mathcal{M}\Delta_{h}\right)^{-1}\tilde{\phi}^{n+1}$ 
and $2\left(-\mathcal{M}\Delta_{h}\right)^{-1}\tilde{\rho}^{n+1}$, respectively, yields
\begin{align*}
&\frac{1}{\mathcal{M}\Delta t}\left(
\|\tilde{\phi}^{n+1}\|_{-1,h}^{2}-\|\tilde{\phi}^{n}\|_{-1,h}^{2}
+\|\tilde{\phi}^{n+1}-\tilde{\phi}^{n}\|_{-1,h}^{2}
+\|\tilde{\rho}^{n+1}\|_{-1,h}^{2}-\|\tilde{\rho}^{n}\|_{-1,h}^{2}
+\|\tilde{\rho}^{n+1}-\tilde{\rho}^{n}\|_{-1,h}^{2}\right)\nonumber\\
&+2\eta^{2}\langle\tilde{\phi}^{n+1},\Delta_{h}^{2}\tilde{\phi}^{n+1}\rangle_{\Omega}
-2\varepsilon\langle\tilde{\phi}^{n+1},\Delta_{h}\tilde{\phi}^{n+1}\rangle_{\Omega}
-2\xi\langle\tilde{\rho}^{n+1},\Delta_{h}\tilde{\rho}^{n+1}\rangle_{\Omega}\nonumber\\
&+\frac{2}{\varepsilon}\left\langle\tilde{\phi}^{n+1}, (\Phi^{n+1}_{N}-\frac{1}{2} )^{3}
- (\phi^{n+1}-\frac{1}{2} )^{3}\right\rangle_{\Omega}
-2\alpha (1+\frac{1}{\delta} ) \langle\tilde{\phi}^{n+1}, \Delta_{h}\tilde{\phi}^{n+1} \rangle_{\Omega} \nonumber\\
&+2\alpha\left\langle\tilde{\phi}^{n+1}, \nabla_{h} \cdot \left( 
 {\cal A} ( \frac{\Psi^{n+1}_{N}}
{{\cal A} |\nabla_{h}^\delta \Phi^{n+1}_{N}|} ) \nabla_{h}\Phi^{n+1}_{N}  
- {\cal A} ( \frac{\rho^{n+1}}
{{\cal A} |\nabla_{h}^\delta \phi^{n+1} |} ) \nabla_{h} \phi^{n+1} \right) \right\rangle_{\Omega} 
\\
& +2\beta \left\langle\tilde{\rho}^{n+1},H^{\prime}\left(\Psi^{n+1}_{N}\right)
-H^{\prime}\left(\rho^{n+1}_{N}\right)\right\rangle_{\Omega} 
 +2\sqrt{2}\alpha\left\langle\tilde{\rho}^{n+1}, \tilde{\rho}^{n+1} \right\rangle_{\Omega}
-2\alpha \left\langle \tilde{\rho}^{n+1}, {\cal A} | \nabla_h^\delta \Phi^{n+1}_{N}| 
- {\cal A} | \nabla_h^\delta \phi^{n+1}|\right\rangle_{\Omega}\nonumber\\
=&\frac{1}{2\varepsilon} \langle\tilde{\phi}^{n+1},\tilde{\phi}^{n} \rangle_{\Omega}
+\frac{2\alpha}{\delta} \langle\tilde{\phi}^{n+1},\Delta_{h}\tilde{\phi}^{n} \rangle_{\Omega}
+2\alpha (\sqrt{2}-1 )\left\langle\tilde{\rho}^{n+1},\tilde{\rho}^{n}\right\rangle_{\Omega} \\
& 
+\frac{2}{\mathcal{M}} ( \langle\tilde{\phi}^{n+1},\tau_{\phi}^{n} \rangle_{-1,h}
+ \langle \tilde{\rho}^{n+1},\tau_{\rho}^{n} \rangle_{-1,h} ) . 
\end{align*}
The estimate for the terms associated with the surface diffusion is straightforward:
\begin{align*}
\langle\tilde{\phi}^{n+1},\Delta_{h}^{2}\tilde{\phi}^{n+1}\rangle_{\Omega}
&=\|\Delta_{h}\tilde{\phi}^{n+1}\|_{2}^{2},  \quad 
-\langle\tilde{\phi}^{n+1},\Delta_{h}\tilde{\phi}^{n+1}\rangle_{\Omega}
 =\|\nabla_{h}\tilde{\phi}^{n+1}\|_{2}^{2},\\
-\langle\tilde{\rho}^{n+1},\Delta_{h}\tilde{\rho}^{n+1}\rangle_{\Omega}
&=\|\nabla_{h}\tilde{\rho}^{n+1}\|_{2}^{2}.
\end{align*}
For the nonlinear inner product, we have the following result
\begin{equation}
2\beta \left\langle\tilde{\rho}^{n+1},H^{\prime}\left(\Psi^{n+1}_{N}\right)
-H^{\prime}\left(\rho^{n+1}_{N}\right)\right\rangle_{\Omega}\geq 0,
\end{equation}
due to the fact that the logarithmic function is an increasing function. 
Similarly, the convexity of the nonlinear functional $g_1 (\phi) = h^2 \sum_{i,j=1}^N ( \phi_{i,j} - \frac12 )^4$ and $g_2 (\phi, \rho) = h^2 \sum_{i,j=1}^N (\rho_{i,j} - {\cal A} | \nabla_h^\delta \phi_{i,j} | )^2 + (\sqrt{2} -1) \| \rho \|_2^2 + \frac{1}{\delta} \| \nabla_h \phi \|_2^2$ leads to the following inequalities: 
\begin{align}
&\frac{2}{\varepsilon} \left\langle\tilde{\phi}^{n+1}, (\Phi^{n+1}_{N}-\frac{1}{2} )^{3} - (\phi^{n+1}-\frac{1}{2} )^{3}\right\rangle_{\Omega} \ge 0 , \nonumber \\
& 2\alpha\left\langle \tilde{\phi}^{n+1},\nabla_{h} \cdot \left( 
 {\cal A} ( \frac{\Psi^{n+1}_{N}}
{{\cal A} |\nabla_{h}^\delta \Phi^{n+1}_{N}|} ) \nabla_{h}\Phi^{n+1}_{N}  
- {\cal A} ( \frac{\rho^{n+1}}
{{\cal A} |\nabla_{h}^\delta \phi^{n+1} |} ) \nabla_{h} \phi^{n+1} \right) \right\rangle_{\Omega} 
-2\alpha (1+\frac{1}{\delta} )
 \langle \tilde{\phi}^{n+1}, \Delta_{h}\tilde{\phi}^{n+1} \rangle_{\Omega} 
  \nonumber \\ 
&+2 \sqrt{2}\alpha\left\langle\tilde{\rho}^{n+1}, \tilde{\rho}^{n+1} \right\rangle_{\Omega}
-2\alpha\left\langle\tilde{\rho}^{n+1},{\cal A} | \nabla_h^\delta \Phi^{n+1}_{N}| 
- {\cal A} | \nabla_h^\delta \phi^{n+1}|\right\rangle_{\Omega} \geq 0 . 
\end{align}
Then we arrive at the following estimate: 
\begin{align*}
&\frac{1}{\mathcal{M}\Delta t}\left(
\|\tilde{\phi}^{n+1}\|_{-1,h}^{2}-\|\tilde{\phi}^{n}\|_{-1,h}^{2}
+\|\tilde{\phi}^{n+1}-\tilde{\phi}^{n}\|_{-1,h}^{2}
+\|\tilde{\rho}^{n+1}\|_{-1,h}^{2}-\|\tilde{\rho}^{n}\|_{-1,h}^{2}
+\|\tilde{\rho}^{n+1}-\tilde{\rho}^{n}\|_{-1,h}^{2}\right)\nonumber\\
&+2\eta^{2}\|\Delta_{h}\tilde{\phi}^{n+1}\|_{2}^{2}
+2\varepsilon\|\nabla_{h}\tilde{\phi}^{n+1}\|_{2}^{2}
+2\xi\|\nabla_{h}\tilde{\rho}^{n+1}\|_{2}^{2}\nonumber\\
\leq &\frac{1}{2\varepsilon}\left\langle\tilde{\phi}^{n+1},\tilde{\phi}^{n}\right\rangle_{\Omega}
+\frac{2\alpha}{\delta}\left\langle\tilde{\phi}^{n+1},\Delta_{h}\tilde{\phi}^{n}\right\rangle_{\Omega}
+2\alpha\left(\sqrt{2}-1\right)\left\langle\tilde{\rho}^{n+1},\tilde{\rho}^{n}\right\rangle_{\Omega}\nonumber
\\
&+\frac{2}{\mathcal{M}}\left\langle\tilde{\phi}^{n+1},\tau_{\phi}^{n}\right\rangle_{-1,h}
+\frac{2}{\mathcal{M}}\left\langle\tilde{\rho}^{n+1},\tau_{\rho}^{n}\right\rangle_{-1,h}.
\end{align*}
Meanwhile, for the inner product associated with the concave part, the following inequalities could be derived: 
\begin{subequations}
\begin{align}
\frac{1}{2\varepsilon}\left\langle\tilde{\phi}^{n+1},\tilde{\phi}^{n}\right\rangle_{\Omega}
&\leq\frac{\varepsilon}{2}\|\nabla_{h}\tilde{\phi}^{n+1}\|_{2}^{2}
+\frac{1}{8\varepsilon^{3}}\|\tilde{\phi}^{n}\|_{-1,h}^{2},
\\
2\left(\sqrt{2}-1\right)\alpha\left\langle\tilde{\rho}^{n+1},\tilde{\rho}^{n}\right\rangle_{\Omega}
&\leq\xi\|\nabla_{h}\tilde{\rho}^{n+1}\|_{2}^{2}
+\frac{\alpha^{2}\left(\sqrt{2}-1\right)^{2}}{\xi}\|\tilde{\rho}^{n}\|_{-1,h}^{2},
\\
\frac{2}{\mathcal{M}}\left\langle\tilde{\phi}^{n+1},\tau_{\phi}^{n}\right\rangle_{-1,h}
&\leq\frac{1}{\mathcal{M}}\|\tilde{\phi}^{n+1}\|_{-1,h}^{2}
+\frac{1}{\mathcal{M}}\|\tau_{\phi}^{n}\|_{-1,h}^{2},
\\
\frac{2}{\mathcal{M}}\left\langle\tilde{\rho}^{n+1},\tau_{\rho}^{n}\right\rangle_{-1,h}
&\leq\frac{1}{\mathcal{M}}\|\tilde{\rho}^{n+1}\|_{-1,h}^{2}
+\frac{1}{\mathcal{M}}\|\tau_{\rho}^{n}\|_{-1,h}^{2},
\\
-\frac{2\alpha}{\delta}\left\langle\tilde{\phi}^{n+1},\Delta_{h}\tilde{\phi}^{n}\right\rangle_{\Omega}
&=\frac{2\alpha}{\delta}\left\langle\nabla_{h}\tilde{\phi}^{n+1},
\nabla_{h}\tilde{\phi}^{n}\right\rangle_{\Omega}\nonumber
\\
&\leq\frac{\varepsilon}{2}\|\nabla_{h}\tilde{\phi}^{n+1}\|_{2}^{2}
+\frac{2\alpha^{2}}{\varepsilon\delta^{2}}\|\nabla_{h}\tilde{\phi}^{n}\|_{2}^{2}\nonumber
\\
&\leq\frac{\varepsilon}{2}\|\nabla_{h}\tilde{\phi}^{n+1}\|_{2}^{2}
+\frac{\alpha^{2}}{\varepsilon\delta^{2}}\|\tilde{\phi}^{n}\|_{-1}^{\frac{2}{3}}
\|\Delta_{h}\tilde{\phi}^{n}\|_{2}^{\frac{4}{3}}\nonumber
\\
&\leq\frac{\varepsilon}{2}\|\nabla_{h}\tilde{\phi}^{n+1}\|_{2}^{2}
+\frac{8\alpha^{6}}{\eta^{4}\varepsilon^{3}\delta^{6}}\|\tilde{\phi}^{n}\|_{-1,h}^{2}
+\eta^{2}\|\Delta_{h}\tilde{\phi}^{n}\|_{2}^{2}.
\end{align}
\end{subequations}
Therefore, we obtain 
\begin{align}
&\frac{1}{\mathcal{M}\Delta t}\left(
\|\tilde{\phi}^{n+1}\|_{-1,h}^{2}-\|\tilde{\phi}^{n}\|_{-1,h}^{2}
+\|\tilde{\phi}^{n+1}-\tilde{\phi}^{n}\|_{-1,h}^{2}
+\|\tilde{\rho}^{n+1}\|_{-1,h}^{2}-\|\tilde{\rho}^{n}\|_{-1,h}^{2}
+\|\tilde{\rho}^{n+1}-\tilde{\rho}^{n}\|_{-1,h}^{2}\right)\nonumber
\\
&+2\eta^{2}\|\Delta_{h}\tilde{\phi}^{n+1}\|_{2}^{2}
+2\varepsilon\|\nabla_{h}\tilde{\phi}^{n+1}\|_{2}^{2}
+2\xi\|\nabla_{h}\tilde{\rho}^{n+1}\|_{2}^{2}\nonumber
\\
\leq &\frac{\varepsilon}{2}\|\nabla_{h}\tilde{\phi}^{n+1}\|_{2}^{2}
+\frac{1}{8\varepsilon^{2}}\|\tilde{\phi}^{n}\|_{-1,h}^{2}
+\xi\|\nabla_{h}\tilde{\rho}^{n+1}\|_{2}^{2}
+\frac{\alpha^{2}\left(\sqrt{2}-1\right)^{2}}{\xi}\|\tilde{\rho}^{n}\|_{-1,h}^{2}
+\frac{1}{\mathcal{M}}\|\tilde{\phi}^{n+1}\|_{-1,h}^{2}
+\frac{1}{\mathcal{M}}\|\tau_{\phi}^{n}\|_{-1,h}^{2}\nonumber
\\
&+\frac{1}{\mathcal{M}}\|\tilde{\rho}^{n+1}\|_{-1,h}^{2}
+\frac{1}{\mathcal{M}}\|\tau_{\rho}^{n}\|_{-1,h}^{2}
+\frac{\varepsilon}{2}\|\nabla_{h}\tilde{\phi}^{n+1}\|_{2}^{2}
+\frac{8\alpha^{6}}{\eta^{4}\varepsilon^{3}\delta^{6}}\|\tilde{\phi}^{n}\|_{-1,h}^{2}
+\eta^{2}\|\Delta_{h}\tilde{\phi}^{n}\|_{2}^{2} ,
\end{align}
which in turn gives 
\begin{align}
&\frac{1}{\mathcal{M}\Delta t}\left(
\|\tilde{\phi}^{n+1}\|_{-1,h}^{2}-\|\tilde{\phi}^{n}\|_{-1,h}^{2}
+\|\tilde{\rho}^{n+1}\|_{-1,h}^{2}-\|\tilde{\rho}^{n}\|_{-1,h}^{2}
\right) \nonumber 
\\
&+2\eta^{2}\|\Delta_{h}\tilde{\phi}^{n+1}\|_{2}^{2}
+\varepsilon\|\nabla_{h}\tilde{\phi}^{n+1}\|_{2}^{2}
+\xi\|\nabla_{h}\tilde{\rho}^{n+1}\|_{2}^{2}\nonumber
\\
\leq &\left(\frac{1}{8\varepsilon^{2}}+\frac{8\alpha^{6}}{\eta^{4}\varepsilon^{3}\delta^{6}}\right)
\|\tilde{\phi}^{n}\|_{-1,h}^{2}
+\frac{\alpha^{2}\left(\sqrt{2}-1\right)^{2}}{\xi}\|\tilde{\rho}^{n}\|_{-1,h}^{2}
+\eta^{2}\|\Delta_{h}\tilde{\phi}^{n}\|_{2}^{2}\nonumber
\\
&+\frac{1}{\mathcal{M}}\|\tilde{\phi}^{n+1}\|_{-1,h}^{2}
+\frac{1}{\mathcal{M}}\|\tau_{\phi}^{n}\|_{-1,h}^{2}
+\frac{1}{\mathcal{M}}\|\tilde{\rho}^{n+1}\|_{-1,h}^{2}
+\frac{1}{\mathcal{M}}\|\tau_{\rho}^{n}\|_{-1,h}^{2}. \label{convergence-f} 
\end{align}
Finally, an application of a discrete Gronwall inequality results in the desired convergence estimate:
\begin{equation}
\|\tilde{\phi}^{n+1}\|_{-1,h}
+\|\tilde{\rho}^{n+1}\|_{-1,h}
+\left(\mathcal{M}\Delta t\sum_{k=0}^{n+1}
\left(\eta^{2}\|\Delta_{h}\tilde{\phi}^{k}\|_{2}^{2}
+\varepsilon\|\nabla_{h}\tilde{\phi}^{k}\|_{2}^{2}
+\xi\|\nabla_{h}\tilde{\rho}^{k}\|_{2}^{2}\right)\right)^{\frac{1}{2}}
\leq C\left(\Delta t+h^{2}\right) , 
\end{equation}
where $C>0$ is independent of $\Delta t, h$ and $n$. This completes the proof of Theorem~\ref{thm: convergence}. 
\end{proof}

\begin{remark} 
In the application of the discrete Gronwall inequality, we see that the growth constants for $\|\tilde{\phi}^{n}\|_{-1,h}^{2}$ and $\|\tilde{\rho}^{n}\|_{-1,h}^{2}$ terms, given by $\frac{1}{8\varepsilon^{2}}+\frac{8\alpha^{6}}{\eta^{4}\varepsilon^{3}\delta^{6}}$ and $\frac{\alpha^{2}\left(\sqrt{2}-1\right)^{2}}{\xi}$, respectively, depend singularly on $\varepsilon$, $\eta$, $\delta$ and $\xi$. In turn, it would be reasonable to require that   
$$
   \Big( \frac{1}{8\varepsilon^{2}}+\frac{8\alpha^{6}}{\eta^{4}\varepsilon^{3}\delta^{6}} \Big) {\cal M} \Delta t \le 1 ,  \quad 
   \frac{\alpha^{2}\left(\sqrt{2}-1\right)^{2}}{\xi} {\cal M} \Delta t \le 1 , 
$$
so that a singular convergence constant is avoided. In other words, the time step size $\Delta t$ should be bounded by a given constant, dependent on $\varepsilon$, $\eta$, $\delta$ and $\xi$, to present a singular convergence constant at a theoretical level; this requirement refers to the condition that ``provided that $\Delta t$ is sufficiently small" in the statement of Theorem~\ref{thm: convergence}. Meanwhile, such a requirement is only associated with a theoretical analysis, and this requirement may not be necessary in the practical computations to preserve a numerical convergence.     
\end{remark}

\begin{remark} 
As the regularization parameter $\delta \to 0$, the positivity-preserving property and the energy stability estimates, as established in Theorems~\ref{thm: positivity} and \ref{thm: energy stability}, are still valid. In fact, these two theoretical properties are available even with $\delta=0$. On the other hand, the optimal rate convergence estimate, as established in Theorem~\ref{thm: convergence}, is only available for a fixed $\delta >0$, due to the singularly-$\delta$-dependent convergence constant appearing in~\eqref{convergence-f}. In other words, the convergence constant in Theorem~\ref{thm: convergence} depends singularly on $\delta$, and such a convergence estimate would not be theoretically justified as $\delta \to 0$, although the numerical convergence has also been verified in various numerical experiments.  
\end{remark}

\begin{remark} 
In the proposed numerical scheme~\eqref{cs_1st}, we take $M_1 = M_2 = {\cal M}$, and $M (\rho) \equiv 1$, for simplicity of presentation. In case of a $\rho$-dependent mobility function $M (\rho)$, the positivity-preserving property and energy stability are still valid, as long as $M (\rho) >0$ is available at a point-wise level. Meanwhile, the corresponding convergence analysis and error estimate are expected to face certain theoretical difficulties in the case of a non-constant mobility function, due to the highly nonlinear and singular nature of the chemical potential. The theoretical justification of this convergence analysis will be left to the future works, and some techniques of rough error estimate and refined error estimate, as reported in a recent work~\cite{LiuC20a} to analyze the non-constant-mobility Poisson-Nernst-Planck system, may have to be applied in this future work.  
\end{remark} 

\section{Numerical experiments} \label{sec: numerical results} 

In this section, we preform a few two-dimensional numerical simulations using the proposed scheme $\eqref{cs_1st}$. The mass conservation, energy decay, positivity of the numerical solution, as well as the numerical accuracy, will be demonstrated in these computations. To achieve this goal, we will present two numerical examples with different initial conditions. 

\subsection{Accuracy test}

Here, we take the domain $\Omega=(0,8)^2$, and choose the parameters as follows 
\begin{align*}
&\varepsilon=0.05,\quad \alpha=0.001,\quad \beta=0.02,\quad \delta=0.001, \\
&\eta=0.05,\quad \xi=0.05,\quad M_{1}=0.01,\quad M_{2}=0.01. 
\end{align*}
The initial data are set as 
\begin{equation}\label{CS:test}
\left\{
\begin{aligned}
&\phi_{0}(x,y)=0.5+0.2\cos\frac{4\pi x}{8}\cos \frac{4\pi y}{8},\\
&\rho_{0}(x,y)=0.5+0.2\sin\frac{4\pi x}{8}\sin \frac{4\pi y}{8}.
\end{aligned}
\right.
\end{equation}
It is obvious that the initial data are subject to periodic boundary condition. 
This example is designed to study the numerical accuracy in time and space. 
In order to test the first order convergence rate in time and second order convergence rate in space, we use a linear refinement pate, i.e. $\Delta t = Ch^{2}$, $C = 0.01$. The global error is expected to be $O(\Delta t)+O(h^{2})=O(h^{2})$ under the discrete $L^{2}$ norm. Since an exact solution is not available, we compute the Cauchy difference instead of directly calculating the numerical error, which is defined as $\delta_{u} = u_{h_{f}}-\mathcal{I}_{c}^{f}(u_{h_{c}})$, where $\mathcal{I}_{c}^{f}$ is a bilinear interpolation operator. 
This requires a relatively coarse solution, parametrized by $h_{c}$, 
and a relatively fine solution, parametrized by $h_{f}$, where $h_{c}=2h_{f}$, at the same final time. The discrete $L^{2}$ norms of Cauchy difference and the convergence rates are displayed in Table $\ref{CS:error}$. These results confirm the expected convergence rate. 
\begin{table}[h!]
    \centering
    \caption{The discrete $L^{2}$ error and convergence rate at $t = 0.1$ 
    with initial data $\eqref{CS:test}$ and the given parameters.}\label{CS:error}
    \vskip 0.2cm
    \begin{tabular}{lcccc}
    \toprule
    Grid sizes    &Error$(\phi)$ & Rate & Error$(\rho)$ &Rate\\
    \midrule
    $16\times 16$    & 1.93E--01 &  --  & 1.88E--01 &  -- \\
    $32\times 32$    & 5.07E--02 & 1.93 & 4.86E--02 & 1.95\\
    $64\times 64$    & 1.28E--02 & 1.98 & 1.23E--02 & 1.99\\
    $128\times 128$  & 3.21E--03 & 2.00 & 3.07E--03 & 2.00\\
    $256\times 256$  & 8.04E--04 & 2.00 & 7.68E--04 & 2.00\\
    \bottomrule
    \end{tabular}
    \end{table}

\subsection{Spinodal decomposition}
In this example, we study the phase separation phenomenon, so called spinodal decomposition. Usually we describe this process as a thermal quench, which is considered that an initially homogeneous mixture is thrust into a two-phase region. In this case, the spinodal decomposition occurs and leads the system from the homogeneous to two-phase state. 
We take the domain as $\Omega = (0,2\pi)^2$. The initial data are given by  
\begin{equation}\label{example: spinodal}
\left\{
\begin{aligned}
&\phi_{0}(x,y)=0.4+0.1\mbox{rand}(x,y), \\
&\rho_{0}(x,y)=0.4+0.1\mbox{rand}(x,y), 
\end{aligned}
\right.
\end{equation}
where rand$(x,y)$ is a random number in  $[-1,1]$ and has zero mean. The parameters are chosen as follows 
\begin{equation}\label{eqn: parameters}
\varepsilon = 0.02,~\alpha = 0.02,~\beta = 0.02,~\eta = 0.02,~\delta = 0.01,~\xi = 0.02,~\mathcal{M} = 0.01. 
\end{equation} 
From Figures $\ref{fig: spinodal t=0}$ to $\ref{fig: spinodal t=700}$, we display the snapshots of coarsening dynamics. Initially, the two fluids are well mixed, and they sooner start to decompose and accumulate. We observe that a relatively high value of the concentration variable $\rho$ gathers at the interface between the two different fluids. A monotone decay evolution of the physical energy is illustrated in Figure $\ref{fig: energy}$. 

\begin{figure}[h]
    \centering
    \subfigure[$\phi(t=0)$]{
    \includegraphics[width=0.4\textwidth]{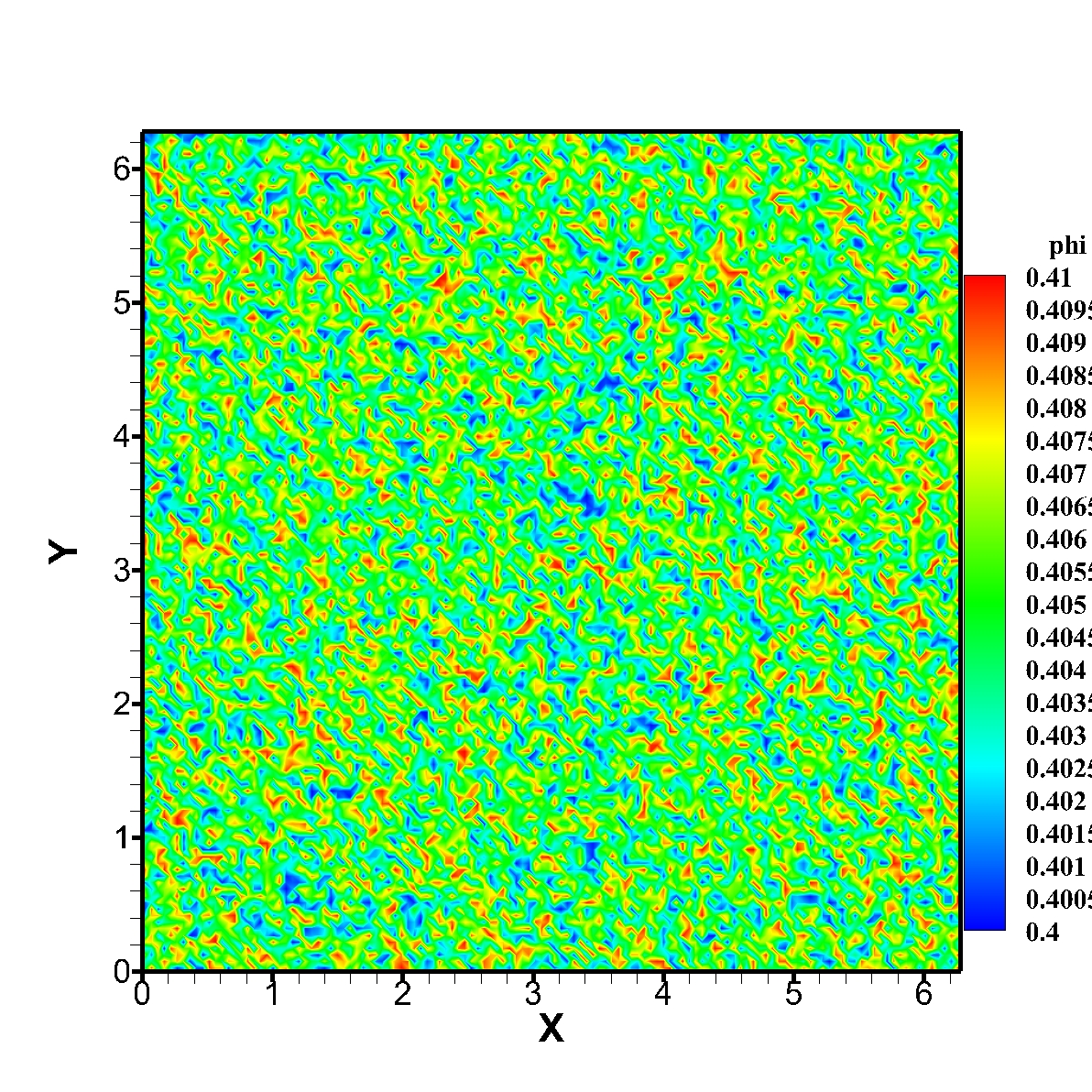} 
}
    \subfigure[$\rho(t=0)$]{
    \includegraphics[width=0.4\textwidth]{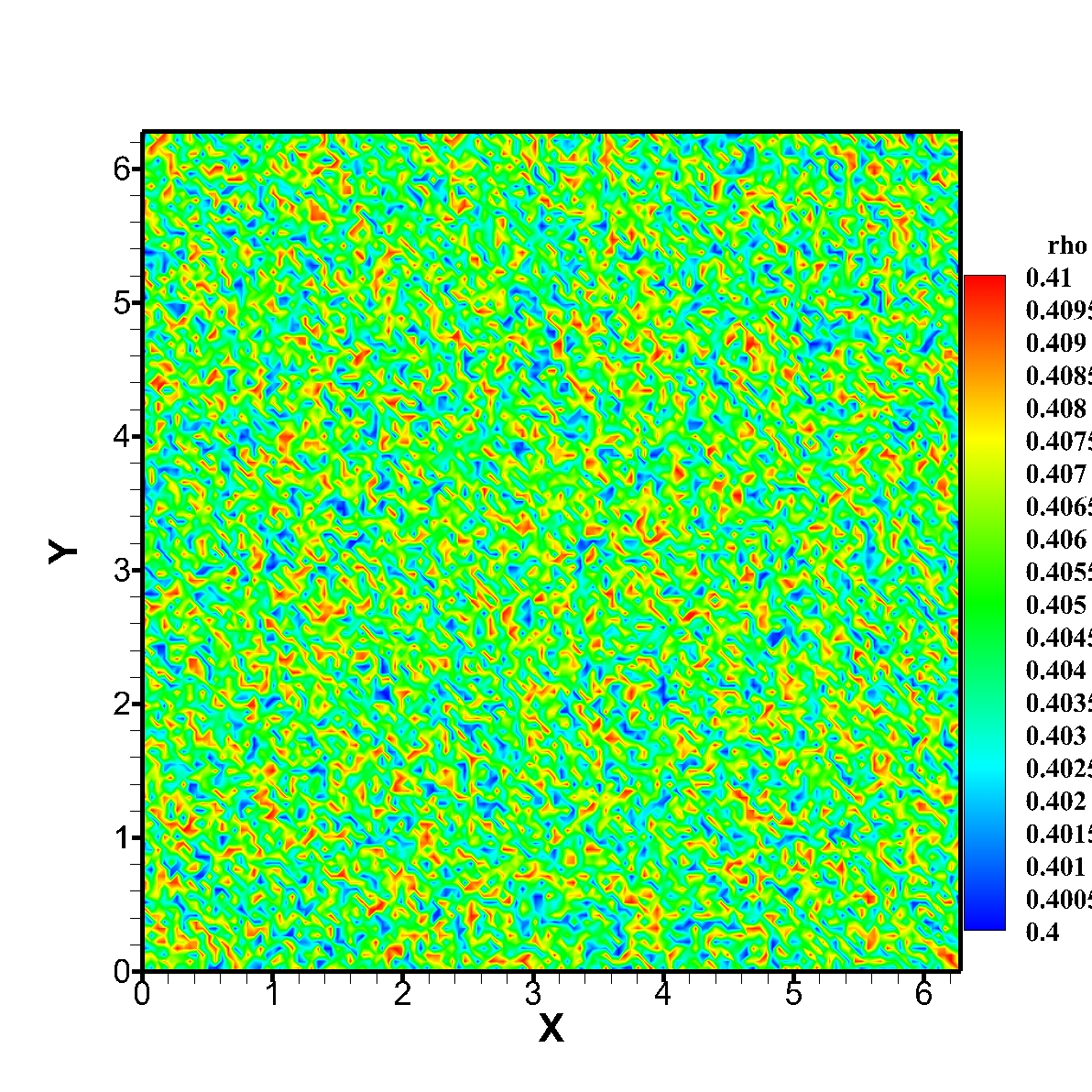}
}
    \caption{Snapshots of the phase variables $\phi$ and $\rho$, taken at $t=0$ for Example $\ref{example: spinodal}$.}
    \label{fig: spinodal t=0}
\end{figure}

\begin{figure}[h]
    \centering
    \subfigure[$\phi(t=0.5)$]{
    \includegraphics[width=0.4\textwidth]{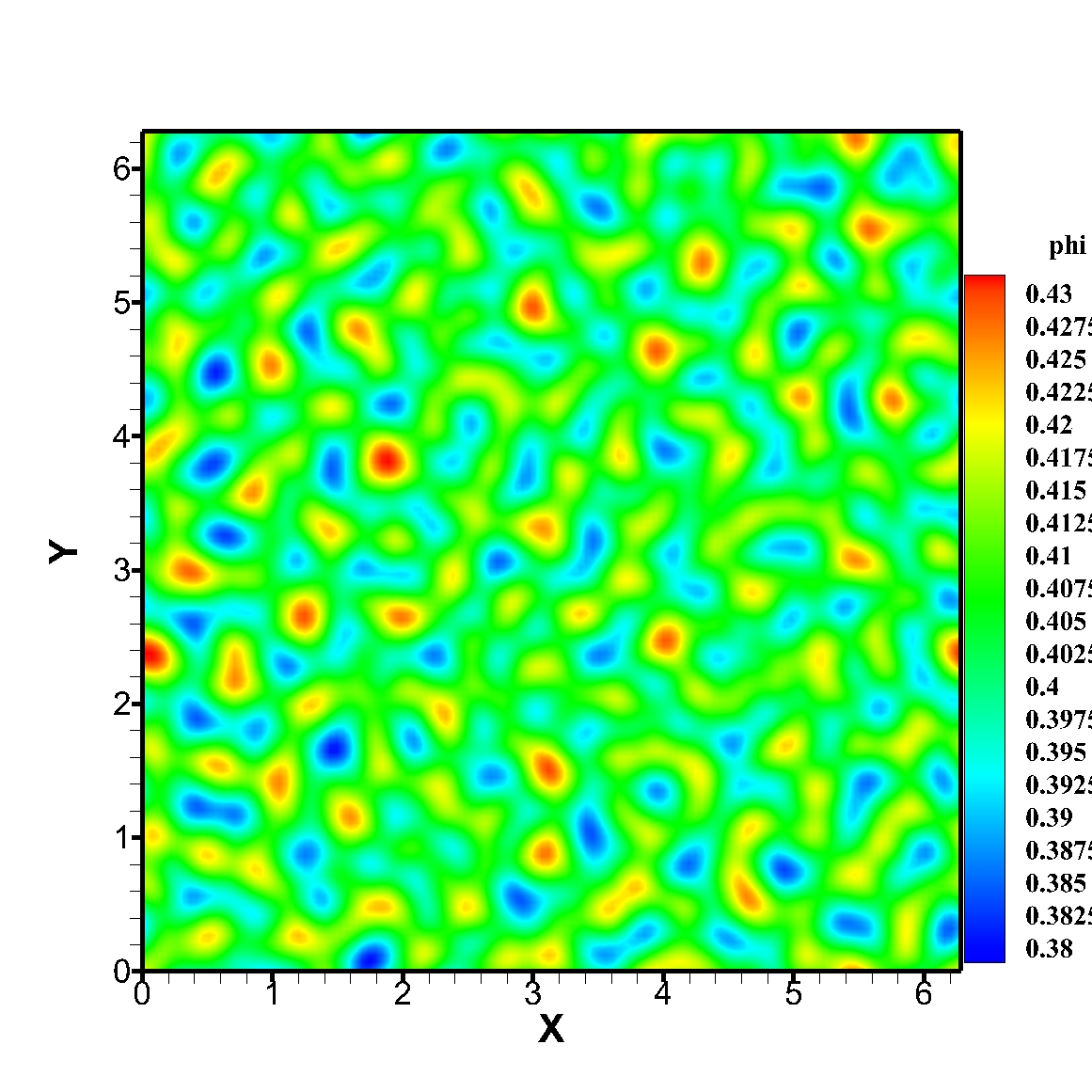} 
}
    \subfigure[$\rho(t=0.5)$]{
    \includegraphics[width=0.4\textwidth]{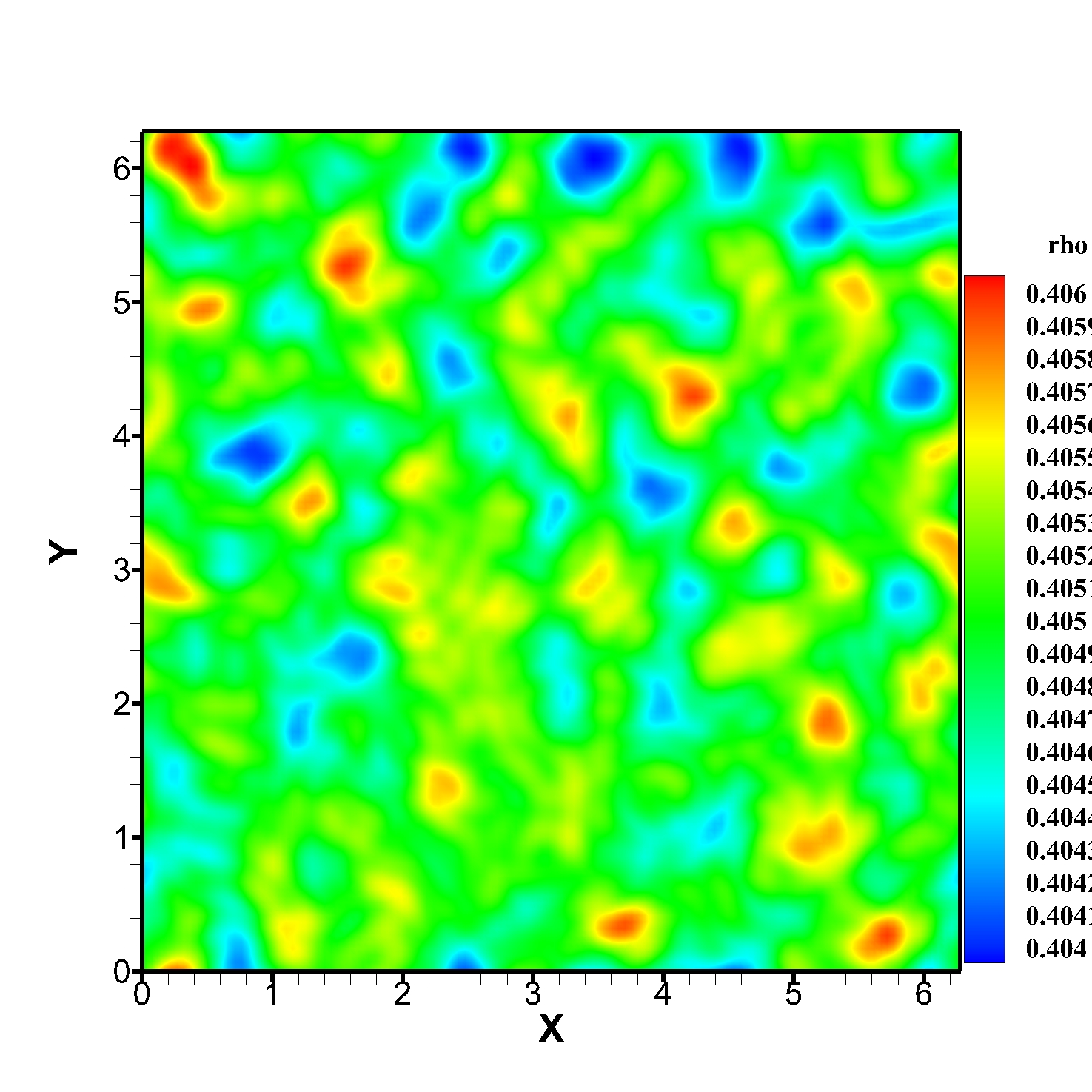}
}
    \caption{Snapshots of the phase variables $\phi$ and $\rho$, taken at $t=0.5$ for Example $\ref{example: spinodal}$.}
    \label{fig: spinodal t=0.5}
\end{figure}

\begin{figure}[h]
    \centering
    \subfigure[$\phi(t=2)$]{
    \includegraphics[width=0.4\textwidth]{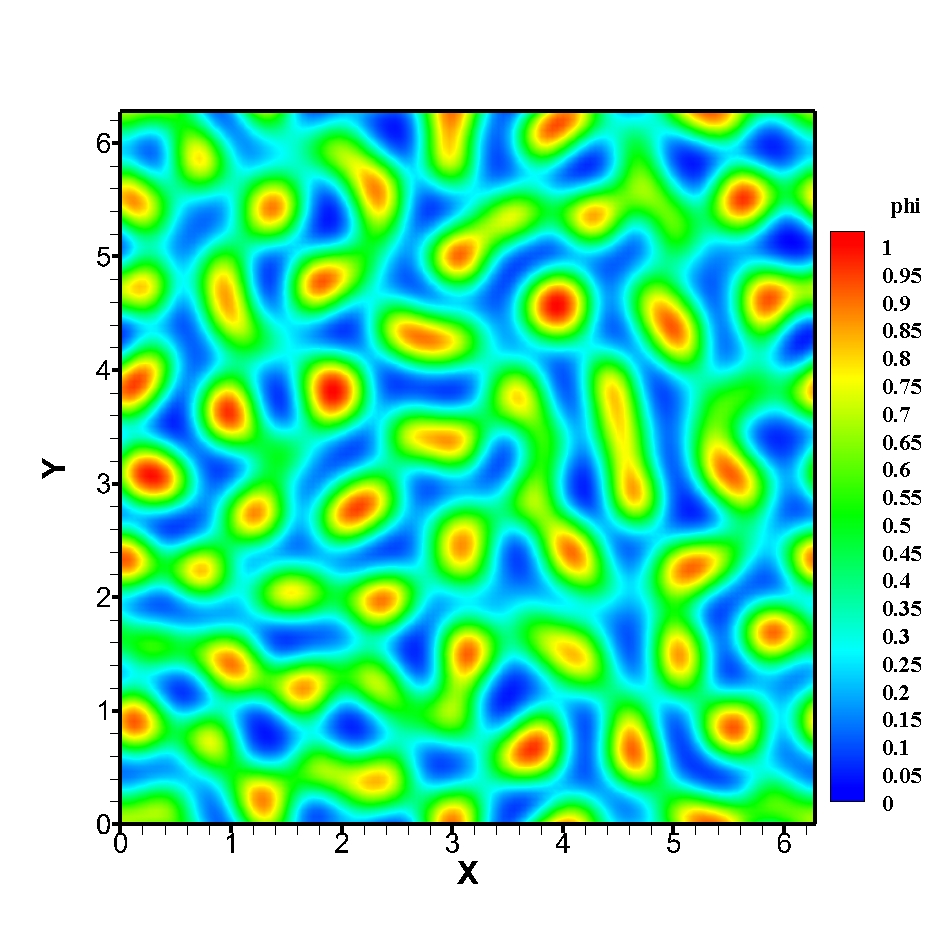} 
}
    \subfigure[$\rho(t=2)$]{
    \includegraphics[width=0.4\textwidth]{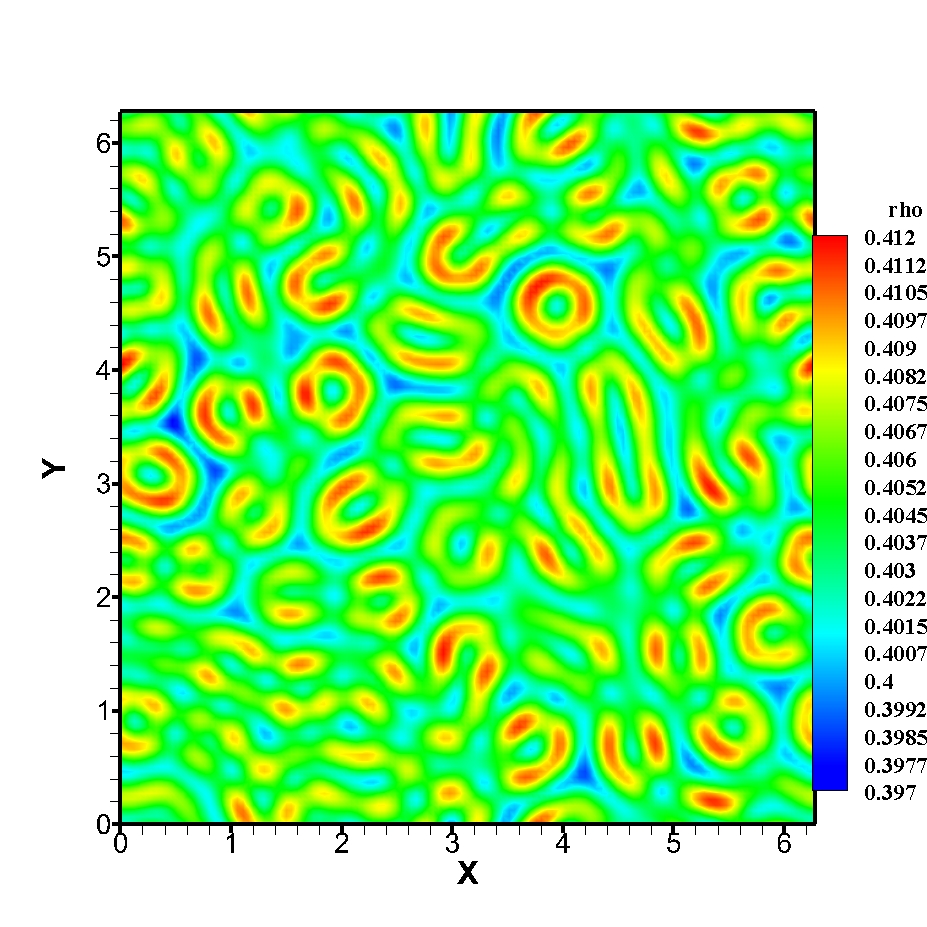}
}
    \caption{Snapshots of the phase variables $\phi$ and $\rho$, taken at $t=2$ for Example $\ref{example: spinodal}$.}
    \label{fig: spinodal t=2}
\end{figure}

\begin{figure}[h]
    \centering
    \subfigure[$\phi(t=10)$]{
    \includegraphics[width=0.4\textwidth]{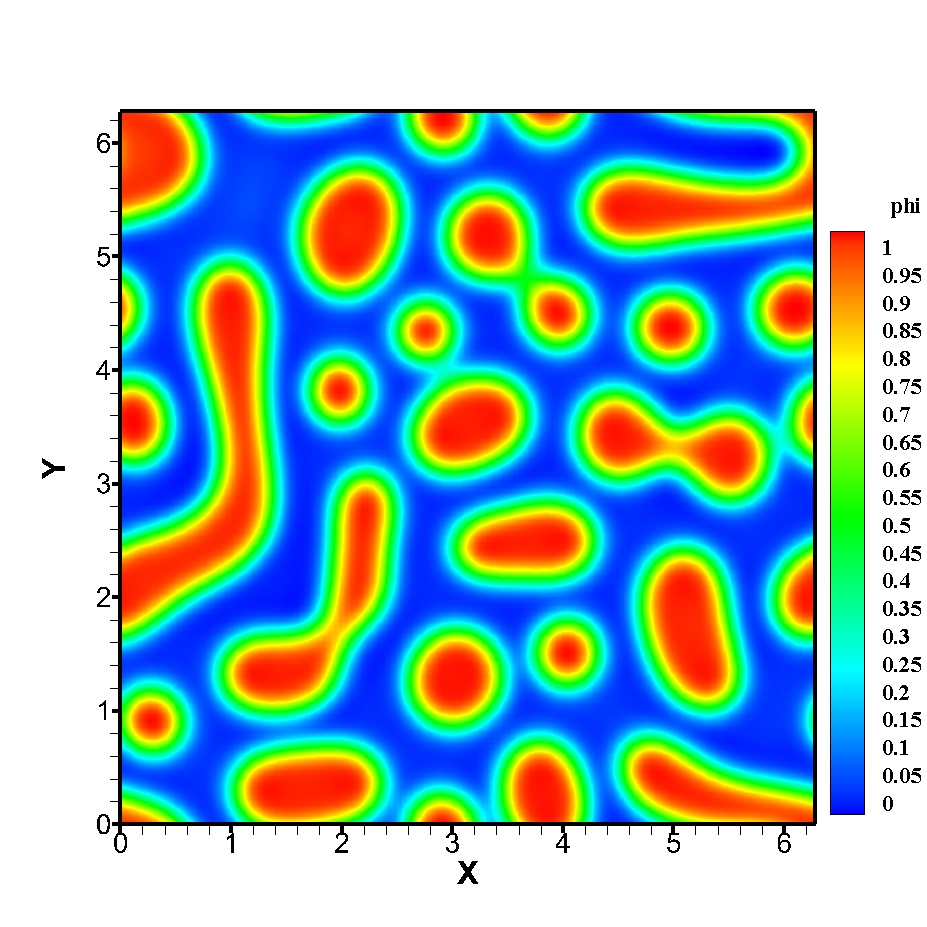} 
}
    \subfigure[$\rho(t=10)$]{
    \includegraphics[width=0.4\textwidth]{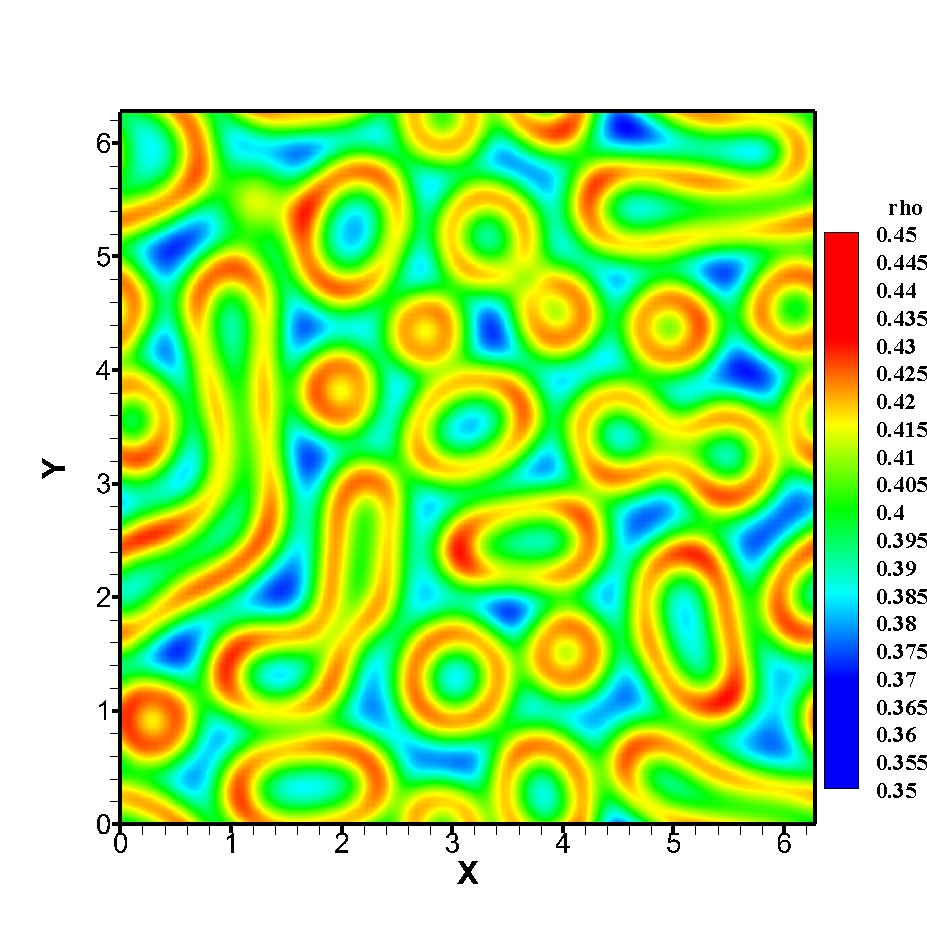}
}
    \caption{Snapshots of the phase variables $\phi$ and $\rho$, taken at $t=10$ for Example $\ref{example: spinodal}$.}
    \label{fig: spinodal t=10}
\end{figure}

\begin{figure}[h]
    \centering
    \subfigure[$\phi(t=40)$]{
    \includegraphics[width=0.4\textwidth]{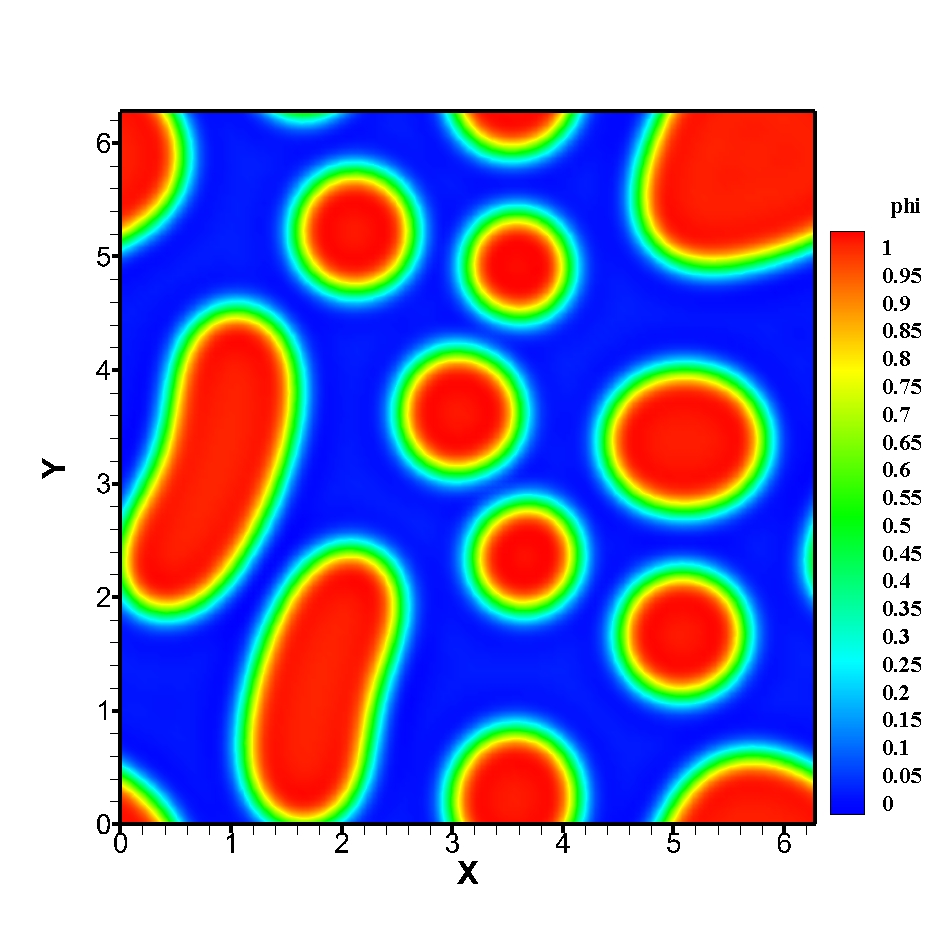} 
}
    \subfigure[$\rho(t=40)$]{
    \includegraphics[width=0.4\textwidth]{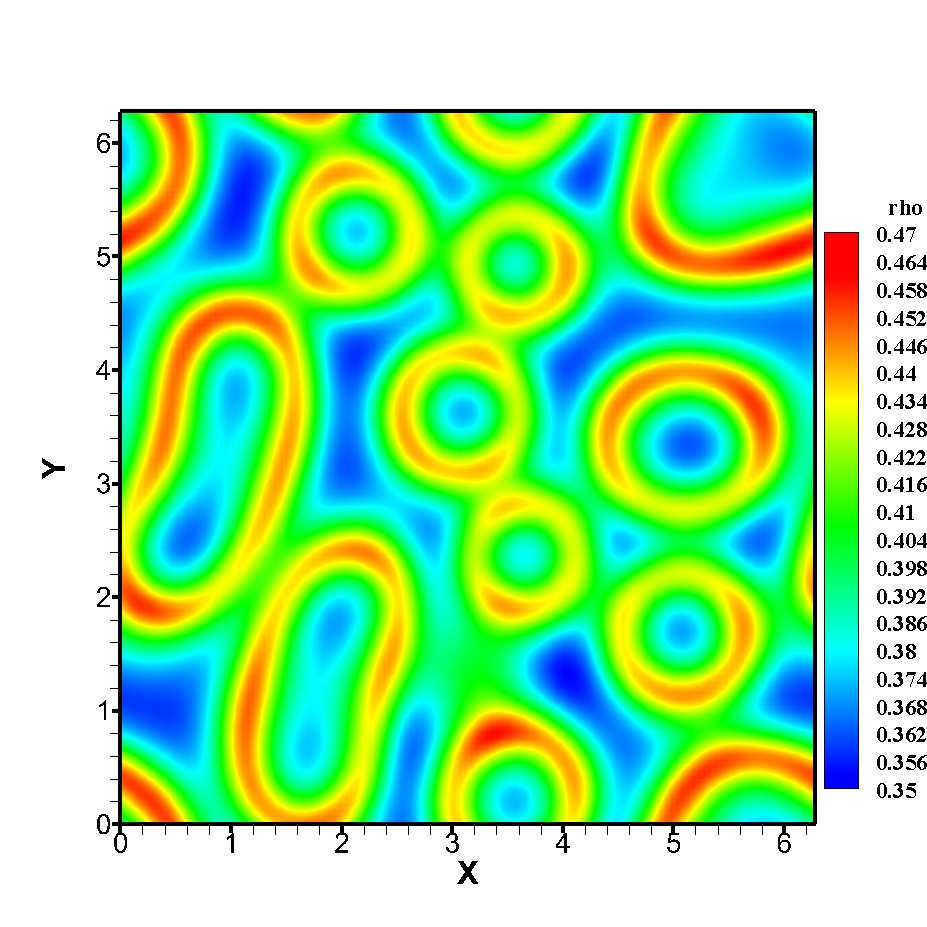}
}
    \caption{Snapshots of the phase variables $\phi$ and $\rho$, taken at $t=40$ for Example $\ref{example: spinodal}$.}
    \label{fig: spinodal t=40}
\end{figure}

\begin{figure}[h]
    \centering
    \subfigure[$\phi(t=700)$]{
    \includegraphics[width=0.4\textwidth]{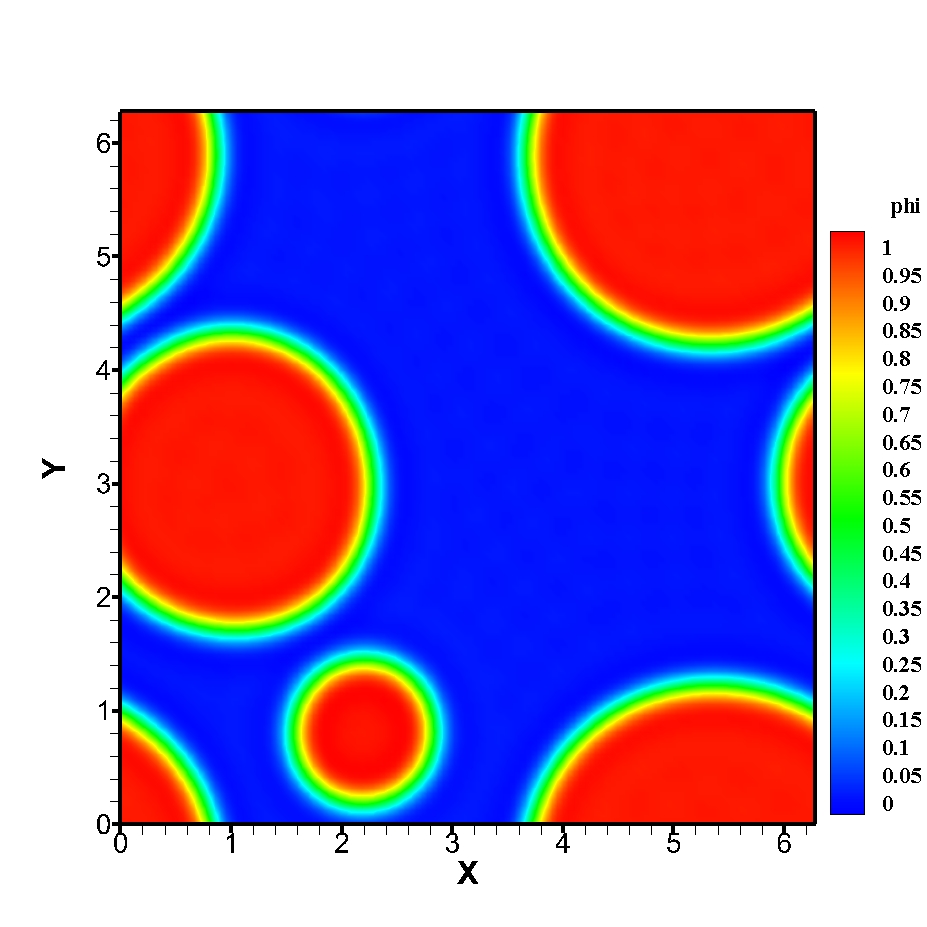} 
}
    \subfigure[$\rho(t=700)$]{
    \includegraphics[width=0.4\textwidth]{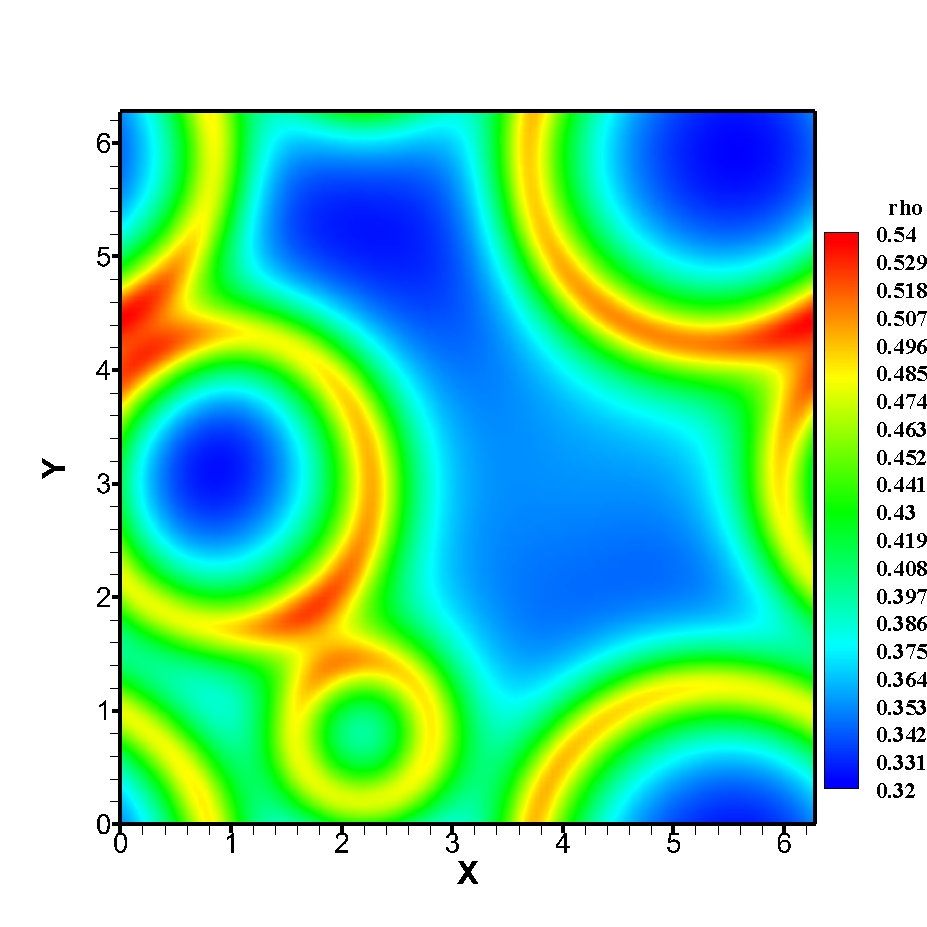}
}
    \caption{Snapshots of the phase variables $\phi$ and $\rho$, taken at $t=700$ for Example $\ref{example: spinodal}$.}
    \label{fig: spinodal t=700}
\end{figure}

\begin{figure}[h]
    \centering
    \includegraphics[width=1\textwidth]{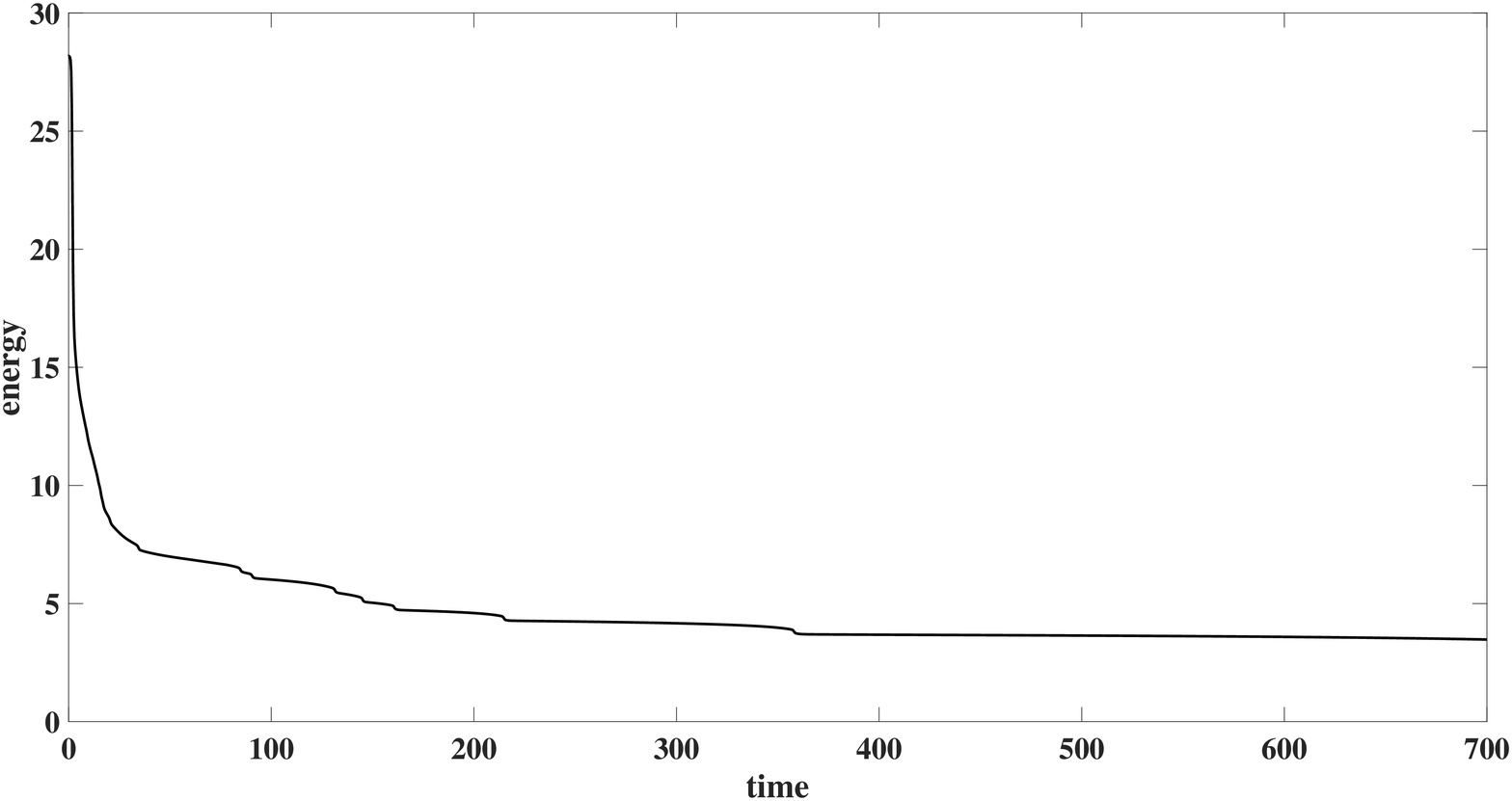}
    \caption{Time evolution of the free energy functional for Example $\ref{example: spinodal}$, which shows a monotone decay.}
    \label{fig: energy}
\end{figure}

\section{Conclusions}  \label{sec: conclusion}

In this paper, we propose and analyze a first order (in time) accurate, convex splitting scheme for the binary fluid-surfactant phase field model. 
The multi-phase structure and the singularity associated with the 1-Laplacian part makes the whole system very challenging, at both the theoretical and numerical levels. 
To overcome this subtle difficulty, we make an observation for a non-standard convex-concave decomposition of the free energy. 
In addition, the singular nature of the logarithmic function around the limit values prevents the numerical solution approaching these limit values, 
so that the positivity property is preserved for the numerical scheme. 
As a result, the convex structure of the implicit part guarantees the unique solvability and energy stability of the proposed numerical scheme. 
Furthermore, an optimal rate convergence analysis is carefully derived, which is the first such result in this area. 
Some numerical experiments are performed to validate the accuracy and energy stability of the proposed scheme.

\section{Acknowledgements}

The authors greatly appreciate many helpful discussions with Professor Hui Zhang, in particular for his insightful suggestions and comments. 
This work is supported in part by the grants NSF DMS-2012669 (C.~Wang), 
NSFC-11871105, the Science Challenge Project TZ2018002 
and the Fundamental Research Funds for the Central Universities (Z.~Zhang).  

\bibliographystyle{plain}
\bibliography{revision1}
\end{document}